\documentclass[10pt, reqno,final]{article}

\usepackage[T1]{fontenc}		     
\usepackage[utf8]{inputenc}			 
\usepackage[english]{babel}

\usepackage{amsmath}
\usepackage{amssymb}
\usepackage{amsthm}
\usepackage{mathtools}		
\usepackage{mathrsfs}		
\usepackage{eucal}			

\usepackage{braket}			
\usepackage[all,pdf]{xy}		
\usepackage{mparhack}		
\usepackage{relsize}			

\usepackage{a4wide}		

\usepackage{booktabs}		
\usepackage{multirow}		
\usepackage{caption}		
\usepackage{rotating}
\usepackage{subfig}

\usepackage{varioref}		
\usepackage{footmisc}

\usepackage{graphicx}		
\usepackage{epsfig}

\usepackage{enumerate}		
\usepackage[colorlinks=true, linkcolor=black, citecolor=black,
urlcolor=blue]{hyperref}		
\mathtoolsset{showonlyrefs}
\input xy
\xyoption{all}
\usepackage{geometry}

\newcommand{\iomega}[1]{\iota_{#1}}
\newcommand{\Lagr}{\Lambda}
\renewcommand{\=}{\coloneqq}			
\newcommand{\N}{\mathbb{N}}                     
\newcommand{\K}{\mathbb{K}}                     
\newcommand{\Z}{\mathbb{Z}}                     
\newcommand{\U}{\mathbb{U}}		
\newcommand{\R}{\mathbb{R}}                     
\newcommand{\C}{\mathbb{C}}                     
\renewcommand{\set}[2]{\left\{{#1}\mid{#2}\right\}}       
\newcommand{\dd}{\mathrm d}			
\newcommand{\Tr}{\mathrm{Tr\,}}               
\newcommand{\ind}{\mathrm{ind\,}}              
\newcommand{\Id}{I\,}                   
\newcommand{\Graph}{\mathrm{Gr\,}}
\newcommand{\Sp}{{\mathrm{Sp}}}                 
\newcommand{\GL}{{\mathrm{GL}}}                 
\newcommand{\Lin}{{\mathscr{L}\,}}      
\newcommand{\comp}{{\mathcal{K}\,}}      
\newcommand{\trace}{{S_1\,}}      
\newcommand{\SSS}{\mathrm{Sym}}         
\newcommand{\Mat}{\mathrm{Mat}}         

\newcommand{\ispec}{\iota_{\scriptstyle{\mathrm{SP}}}}
\newcommand{\iMor}{\iota_{\scriptstyle{\mathrm{MOR}}}}

\newcommand{\icon}{\iota_{\scriptstyle{\mathrm{PW}}}}
\newcommand{\spfl}{\mathrm{sf\,}}

\renewcommand{\proof}{{\sl Proof.}\hspace{5pt}}   
\newcommand{\Imm}{\mathrm{Im}}
\newcommand{\spec}{\mathfrak{s}}
\newcommand{\iindex}[1]{\mu_{\scriptscriptstyle{\mathrm{Mor}}}\left[#1\right]}
\newcommand{\cfsa}{\mathcal{CF}^{sa}}
\newcommand{\Real}{\mathrm{Re}}
\newcommand{\norm}[1]{\left\| #1 \right\|}			
\newcommand{\opnorm}[2]{\norm{#1}_{\mathscr{L}(#2)}}		
\DeclareMathOperator{\rk}{rank}		
\newcommand{\coiindex}[1]{\mathrm{n_+}\left[#1\right]}
\newcommand{\irel}{I}		
\DeclareMathOperator{\sgn}{sgn}		
\DeclareMathOperator{\diag}{diag}		
\newcommand{\iCLM}{\iota^{\scriptscriptstyle{\mathrm{CLM}}}}
\newcommand{\iMorse}[1]{\iota_{#1}}		
\newcommand{\iMas}[1]{\iota_{#1}}
\newcommand{\coindex}{\mathrm{n_+}}
\newcommand{\iiindex}{\mathrm{n_-}}
\newcommand{\email}[1]{\href{mailto:#1}{\textsf{#1}}}
\newcommand{\coiMor}{\coindex}	

\newtheorem{thm}{\sc Theorem}[section]  
\newtheorem{cor}[thm]{\sc Corollary}       
\newtheorem{lem}[thm]{\sc Lemma}            
\newtheorem{prop}[thm]{\sc  Proposition}     
        
\newtheorem{defn}[thm]{\sc Definition}     
\newtheorem{rem}[thm]{\sc Remark}      
\newtheorem{ex}[thm]{\sc Example}          
\newtheorem{note}[thm]{\sc Notation}

\usepackage{bm}

\newcommand{\trasp}[1]{{#1}^\mathsf{T}}	

\title{Spectral flow, Brouwer degree and Hill's determinant formula}
\author{Alessandro Portaluri
\thanks{The author is partially supported by   Prin 2015 ``Variational methods, with applications to
problems in mathematical physics and geometry” No.~$\mathrm{2015KB9WPT\_001}$.}, Li Wu\thanks{The author is  supported by NSFC No.~$\mathrm{11425105}$.}}
\date{\today}
\begin{document}
 \maketitle

\begin{abstract}
In 2005 a new topological invariant defined in terms of the Brouwer degree of a determinant map, was introduced by Musso, Pejsachowicz and the first name author for counting the conjugate points along a semi-Riemannian geodesic. This invariant was defined in terms of a suspension of a complexified family of linear second order Dirichlet boundary value problems. 

In this paper, starting from this result,  we generalize this invariant to a general self-adjoint Morse-Sturm system and we prove a new spectral flow formula.  Finally we discuss the relation between this spectral flow formula and the Hill's determinant formula and we apply this invariant for detecting instability of periodic orbits of a Hamiltonian system. 

\end{abstract}

\tableofcontents


\section*{Introduction}

The investigation of the relation between the variational properties of a critical point and geometrical properties of the solution space of its associated linearized equation at this critical point was firstly recognized by Jacobi who proved  that the  Lagrangian action functional of a regular one dimensional problem is  minimized along  solutions of the Euler-Lagrange equation up to first conjugate point. 

Starting on this result, Morse  worked out a precise formula  between the  {\em Morse index\/} of a critical point of Lagrangian action functional of a regular one dimensional problem and an integer (which somehow generalizes the total number of conjugate points along a geodesic in a Riemannian case)   counted along this curve. Since then,  many generalizations were obtained to higher order ordinary differential operators,  minimal surfaces and elliptic partial differential operators and so on. 

The importance of such a relation is due to its central role played in  a number of questions related to the dynamical properties of  Lagrangian systems as well as in many geometrical problems especially when Duistermaat, in his celebrated and deep paper \cite{Dui76}, was able to prove a precise relation between the Morse index and the Maslov index, namely a  topological invariant intersection index defined  in the Lagrangian Grassmannian manifold. 

Starting from Duistermaat aforementioned paper,  several  different relations were  obtained in the last decades and several applications on  bifurcation theory to Morse-Floer homology or to stability of periodic orbits of Hamiltonian systems, were proven. 

The common denominator of all of these proofs is a sort of  {\em spectral flow formula\/} relating the spectral flow (which is the right substitute to the Morse index in the strongly indefinite situation, which usually appears in  the Hamiltonian case) with the Maslov (intersection) index.

Roughly speaking, it is well-known, in fact,  that  the spectral flow of a path of (bounded)  self-adjoint Fredholm  operators $t \mapsto A_t$  is the integer given by the number of negative eigenvalues of $A_a$ that become positive as the parameter $t$  goes from $a$ to $b$ minus the number of positive eigenvalues of $A_a$ that   become negative. By the topological characterization  of the space of all bounded self-adjoint Fredholm operators and the definition of  the spectral flow it follows  that if one  of (and hence all) the operators in the path have a finite Morse index, then the spectral flow of  a path $A$ is nothing but the difference between the Morse indexes at the end points. Thus spectral flow appears to be the right  substitute of the Morse index  in the framework of strongly  indefinite functionals.

In a series of paper (cfr. \cite{PPT04, MPP05, MPP07, Por09, Por11,  PW14c, HP17, HPY19} and references therein), such index theorems were used  for detecting the bifurcation along a trivial branch of critical points for family of functionals (maybe in the strongly indefinite situation).  Other interesting applications were provided in the last decades for  establishing the  linear stability and instability properties for critical points of regular and singular Lagrangian system (cfr. \cite{BJP14a,BJP14b, HP19a, HPY19}) or to detect the bifurcation in some elliptic PDE's on bounded domains (cfr. \cite{PW14a, PW14b} and references therein) in which among others authors provided a new proof to the Morse-Smale  index theorem \cite{Sma65,Sma67}.

Another interesting application of a Morse type index theorem is 
to graduate the Morse-Floer homology complex (cfr.  \cite{APS08} and references therein).
In general the spectral flow depends upon the homotopy class of the whole path and  not only on its ends.   What we  propose in this paper is  
\begin{itemize}
\item[-] a definition of a new topological invariant named {\em index-degree\/} and denoted by $\icon$ defined in terms of a line integral of the trace of an operator-valued  one form defined by  a suspension of a complexified family of boundary value problems. 
\item[-] an abstract version of the Morse index theorem which can 
be summarize as the equality between the spectral flow of a path of 
self-adjoint unbounded Fredholm operators  and $\icon$ which generalize the conjugate  index introduced some years ago by author in \cite{MPP05}. 
\end{itemize}
As by-product, we provide a different proof to the Morse index theorem proved in \cite{MPP05}. As the reader recognizes,  $\icon$-index is very close  to other topological invariants  which detect gauge anomalies (cf. \cite{Ati84}). Related ideas in the context of Sturm-Liouville boundary value problems  can be found in \cite{GST96}.

\subsection*{Notation}

At last, for the sake of the reader, let us introduce some common notation that we shall use henceforth without further comments  throughout the paper. 
\begin{itemize}
\item The pair $(H, \langle \cdot, \cdot \rangle)$ will denote a complex separable  Hilbert space
\item $\Lin(H)$ (resp. $\Lin_s(H)$) denotes the  Banach space of linear bounded (resp. bounded an selfadjoint) operators on $H$ with respect to the operator norm $\norm{\cdot}_{\mathscr L(H)}$; $\comp(H)$ (resp. $\comp_s(H)$) denotes the subspace of $\Lin(H)$ of compact (resp. compact and symmetric) operators on $H$. 
\item $\trace(H)$ denotes the set of trace class linear operators.
\item Given $A \in \Lin(H)$ we denote by $A^*$ its adjoint. We let $|A|\=(A^*A)^{1/2}$.
\item $\Mat(N,\K)$ denotes the space of $N \times N$ matrices over the field $\K$. $\SSS(N)$ denotes the space of $N \times N$ real symmetric matrices and by $\SSS^+(N)$ the set of all positive definite   $N \times N$ real symmetric matrices. 
\item  Let  $X\subset \C$ be an open subset of the complex plane and  let $W$ be a normed vector space. With a slight abuse of notation, we say that $F:X \ni z\mapsto F(z) \in W$ is a $\mathscr C^1$-map or a map of regularity class $\mathscr C^1$, if it has regularity class  $\mathscr C^1$ as a map from on  open subset of $\R^2$ into $W$.
\item We denote by $\dd$ the exterior differential and finally  by $\Omega$ the rectangle 	defined by 
\[
\Omega\=\Set{z:=t+is| t \in [0,1] \textrm{ and } s \in [-1,1]}\subseteq \C.
\]

\end{itemize}

\subsection*{Acknowledgements}
The  first name author wishes to thank all faculties and staff of the Mathematics Department at the Shandong University (Jinan) as well as the Department of Mathematics and Statistics at the Queen's University (Kingston)  for providing excellent working conditions during his stay and especially his wife, Annalisa, that  has been extremely supportive of him throughout this entire period  and has made countless sacrifices to help him  getting to this point.


\section{Functional analytic preliminaries}\label{sec:functional-preliminaries}

This section is essentially devoted to recall some elementary results that we shall need throughout the paper as well as to introduce our basic notation.  In particular we shall discuss some basic properties of trace class operator valued one-forms and we shall prove some technical functional analytic results.



 \subsection{Trace class  operators: definitions and basics properties}\label{subsec:trace-class}

 Let $(H, \langle \cdot, \cdot \rangle)$ be a complex separable  Hilbert space,   $\comp(H)$ be the ideal of compact linear operators on $H$ and let $T \in \comp(H)$ and  $T^*$ its adjoint. It is  immediate to check  that  $T^*T$ is a compact self-adjoint non-negative definite linear operator.
 
For $j \in \N^*$, we let  $ \lambda_j(T^* T)$
be the non-zero eigenvalues of $A^*A$ where each eigenvalue is repeated according to its own multiplicity and  we define the  {\em $j$-th singular value\/} or {\em $j$-th s number\/} as the non-negative number given by   $s_j(A)\=\big[\lambda_j(T^*T)^{1/2}\big]$.

We set $ \trace(H)\=\Big\{T\in \comp(H)| \sum_{j=1}^\infty s_j(T)<+\infty\Big\}$. Endowing  $\trace(H)$ with the {\em trace norm\/} given  by  $\|T\|_\trace\= \sum_{j=1}^\infty s_j(T)$, 
then  $\trace(H)$  turns into a Banach space,  termed the space of {\em trace class operators\/} and its elements are   {\em trace class operators\/} on $H$. 
\begin{rem}
We observe that the finite rank operators are dense in $\trace(H)$ with respect to the $\norm{\cdot}_\trace$. Moreover  $\|T\|_\trace \geq \|T\|$. (Cf.  to \cite[Theorem 4.1, Chapter VI, pag.105]{GGK90} for further details). An useful property we shall use throughout the paper is that  if $S \in \Lin(H)$ and $T\in \trace(H)$, then $ST\in \trace(H)$ and $ TS  \in \trace(H)$.	
 \end{rem}
 It is well-known that the trace class $\trace(H)$ is a bilateral ideal contained in the ideal of all compact operators $\comp(H)$. On $\trace(H)$, there is the  well-defined linear functional called {\em trace\/} that will be denoted by $\Tr$ and defined by 
 \[
 \Tr: \trace(H) \to \R \quad \textrm{ defined by } \quad  \Tr(T)\=\sum_{j \in \Z}\langle T e_j, e_j\rangle
 \]
where $(e_j)_{j \in \Z}\subset H$ is an orthonormal basis of $H$. Among all properties of the trace, we recall that 
 \begin{multline}\label{eq:commutativity-trace}
\textrm{ if }\quad ST, TS \in \trace(H)\quad  \Rightarrow\quad  \Tr(ST)= \Tr(TS)\\
	\textrm{and  if }\quad S \in \trace(H)\textrm { and } T\in \mathscr{L}(H)\quad  \Rightarrow \quad  \|ST\|_{ \trace} \leq \|S\|_{ \trace} \|T\|_{\mathscr L(H)} \textrm{ and } \\ \|TS\|_{ \trace}\leq \|S\|_{ \trace} \|T\|_{\mathscr L(H)}.
	\end{multline}
	The next result is the Jacobi's (determinant) formula that we shall need in the proof of the abstract trace forma. Let $z\mapsto S_z$ be a $\mathscr C^1$ map of $N\times N$ complex matrices. Then, the following formula is well-known
\begin{equation}\label{eq:jacobi}
	\dd(\det(S_z))=\Tr (\textrm{adj} (S_z)\dd S_z),
	\end{equation}
	where $\textrm{adj}(\#)$ denotes the {\em adjugate\/} (namely the transpose of the cofactor matrix) of the matrix  $\#$.
  The following result holds. 
	\begin{lem}\label{thm:lemma-3}
		Let $S \in \mathscr C^1\big(X, \GL(N, \C)\big)$. Then we have
		\[
		\Tr\big(\dd S_z S_z^{-1}\big)= \dd \log \det  S_z.
		\]
	\end{lem}
\proof 
For the proof we refer the interested reader to \cite{MN99} and references therein.
\qed
\begin{lem}\label{thm:ffava}
		Let $H,L$ be Hilbert spaces, $A\in\Lin(H,L)$ and let $H_0$ be a finite co-dimensional  subspace of $H$.
				If $A_0\= A|_{H_0}\in \trace(H_0)$, then $A \in \trace(H)$.
	\end{lem}
\proof
		Let $i: H_0 \mapsto H$ and $j:H_0^\perp\to  H$  be the injections. Then $i^*,j^*$ are the orthogonal projections from $H$ to $H_0$ and $H_0^\perp$ respectively.
		We observe that $A|_{H_0}=A i$. 
		Since $|A|=i^*|A|i+i^*|A|j+j^*|A|i+j^*|A|j$, we have
		$A_0^*A_0=i^*A^*Ai=(i^*|A|i)^2+(i^*|A|j)(i^*|A|j)^*\ge (i^*|A|i)^2$. 
		Since $A_0$ is in the trace class, $i^*|A| i$ is also in the trace class and being $H_0 $  finite co-dimensional subspace of $H$, $|A|$ is an operator in the trace class. So A is in the trace class and this concludes the proof. 	
\qed 

\begin{lem}\label{thm:Sobolev-inclusion-trace-class}
		The Sobolev embedding 
		\begin{equation}
		\iota: H^2\big([0,1], \C^N\big)\hookrightarrow  L^2\big([0,1], \C^N\big)
		\end{equation}
		is in the trace class. 
	\end{lem}
	\begin{proof}
		In order to prove that $\iota: H^2\big([0,1], \C^N\big)\hookrightarrow  L^2\big([0,1], \C^N\big)$ is in the trace class, we start to observe that 
		\[
		|\|u\||= \|u\|_{L^2}+\|D^2 u\|_{L^2}
		\]
		is an equivalent norm in $H^2$. Thus we have that the map 
		\[
		D^2: H^2\big ([0,1],\mathbb{C}^N\big)\cap  H^1_0\big([0,1],\mathbb{C}^N\big)\to L^2\big([0,1],\mathbb{C}^N\big)
		\]
		 is an isomorphism.
		We now observe that $D^2: H^2\big (I,\mathbb{C}^N\big)\cap  H^1_0\big(I,\mathbb{C}^N\big)\subset L^2\big(I,\mathbb{C}^N \big)\to L^2\big(I,\mathbb{C}^N\big)$ is in the trace class since it is self-adjoint and the spectrum is given by $\{k^2\pi^2\}_{k\in \mathbb{N}}$.
		
		So $\iota|_{H^2\big(I,\mathbb{C}^N\big)\cap H^1_0\big(I,\mathbb{C}^N\big)}$ is in the trace class. The conclusion now readily follows by using Lemma~\ref{thm:ffava} once observed  that the space  $H^2\big([0,1], \C^N\big)\cap H_0^1\big([0,1], \C^N\big)$ is a finite-codimensional subspace of $H^2\big([0,1], \C^N\big)$. This concludes the proof.
\end{proof}


 \subsection{Trace class  operator valued one-forms}

With a slight abuse of notation, given an open subset  $X\subset \C$, we let $L \in \mathscr C^1\big(X, \trace(H)\big)$ meaning that  $L$ is a  a map of regularity class $\mathscr C^1$  from the   open subset of $X \subset \R^2$ into $\trace(H)$. The next result provides a useful relation between the trace and the exterior differential for an operator valued function. More  precisely, the following result holds. 
\begin{lem}\label{thm:fava}
Let $L \in \mathscr C^1\big(X, \trace(H)\big)$, $F\=\Id+L$ and we assume that for each $z \in X$, $\norm{L_z}_{\Lin(H)} <1$. 
Under the above notation, the following equality holds
\[
\dd \Tr\log F_z=\Tr(\dd F_z\,F_z^{-1}).
\]
\end{lem}
\proof
Since $\opnorm{L_z}{H} <1$,  then we get that $\displaystyle\log F_z=\log(\Id+  L_z) = \sum_{k=1}^\infty (-1)^{k+1} \dfrac{ L_z^k}{k}$ is well-defined. Moreover  $\sum_{k=1}^\infty (-1)^{k+1} \frac{\Tr  L_z^k}{k}$ converges in $\trace(H)$ as directly follows by observing that 
\[
\left\|\sum_{k=m}^n (-1)^{k+1} \dfrac{\Tr  L_z^k}{k}\right\|_{ \trace}\le \left|\sum_{k=m}^n (-1)^{k+1} \dfrac{  \|L_z\|_{\mathscr{L}(H)}^{k-1}}{k}\right|\|L_z\|_{ \trace(H)}.
\] 
By using the  linearity of the trace, we get  
\begin{equation}
\Tr \log(\Id+ L_z) = \sum_{k=1}^\infty (-1)^{k+1} \dfrac{\Tr  L_z^k}{k}
\end{equation}
and by an induction  argument, we get  that  $ \dd L_z^k= \sum_{m=1}^{k} L_z^{m-1} \dd L_z L_z^{k-m}$ for all $z \in X$. Thus, we have 
\begin{multline}\label{eq:1f1}
\Tr \dd L_z^k=  \sum_{m=1}^{k} \Tr\big[ L_z^{m-1} \dd L_z  L_z^{k-m}\big]=\sum_{m=1}^{k} \Tr\big[ L_z^{k-1} \dd L_z\big]= k \Tr\big[L_z^{k-1} \dd L_z\big]= k \Tr\big( \dd L_zL_z^{k-1}\big) 
\end{multline}
where, in the second and last equality of Equation~\eqref{eq:1f1}, we used the commutativity property described in Formula \eqref{eq:commutativity-trace}.
By using once again  the linearity of $\Tr$ and $\dd$, we finally  get 
\begin{multline}
\dd \Tr \log F_z\=\dd \Tr \log(\Id+L_z)= \sum_{k=1}^\infty (-1)^{k+1} \dfrac{\Tr  \dd L_z^k}{k}
\\
= \sum_{k=1}^\infty (-1)^{k+1} \Tr \big(\dd L_z  L_z^{k-1}\big)= \Tr \big[\dd L_z \sum_{k=1}^\infty (-1)^{k+1}L_z^{k-1}\big]= \Tr\big[\dd L_z\,(\Id+ L_z)^{-1}\big]\\=\Tr(\dd F_z\,F_z^{-1})
\end{multline}
where, in the second equality, we commute $d$ and $\sum$ since 
	$\sum_{k=1}^\infty (-1)^{k+1} \Tr \big(\dd L_z  L_z^{k-1}\big)$ converges uniformly in some neighborhood of $z$ (being, in fact, $z\mapsto L_z$  of regularity class $\mathscr{C}^1$ in the trace norm topology). This concludes the proof. 
\qed

\begin{lem}\label{thm:lemma-0-bis}
Under the assumptions given in Lemma \ref{thm:fava}, we get that the spectrum $\spec( F_z)$ is discrete. Moreover, if $\lambda \in \spec(F_z)$,  then for every  $\delta >0$, we get $|\lambda-1| \leq \delta$ for all except a finite number. 
\end{lem}
\proof
By the spectral theory of   compact  operators, it follows that, for every $z \in X$,  $0$ is the only accumulation point for eigenvalues of $L_z$ and hence $1$ is the only possible accumulation point of eigenvalues for $F_z$. Thus, for every $\delta >0$ the eigenvalues out off the  disk $ \Set{\lambda\in \C|\lambda-1|<\delta} \subset \C$ counted with multiplicity are in a finite number. \qed

\begin{lem}\label{lm:inv_exact}
Let 	$X\subset \C$ be an open and convex set containing $0$, $W \hookrightarrow H$ be a dense injection  and we assume that
\begin{itemize}
\item   $z\mapsto G_z \in \Lin (W,H)$ is  of class $\mathscr C^1$;
\item there is a open set $Y \subset  X$ containing $0$ and such that $G_z$ is invertible for $z \in Y$;
\item $z \mapsto (G_z-G_0)G_0^{-1}\in \mathscr{C}^1(Y,\trace(H))$
\end{itemize}
Then we get that  $\Tr \dd G_z G_z^{-1}$ is a closed  one form on $X$.

Furthermore,  if  $G_z=G_0+C_z$ with $C_z\in \mathscr C^1(X,\Lin(H))$ and $G_0^{-1}\in  \trace(H)$, then  $\Tr(\dd G_zG_z^{-1})$ is a closed  one form on $Y$.
\end{lem}
\proof
For the sake of the reader, we split the proof into some steps. \\
{\bf First step.\/}  $\Tr \dd G_z G_z^{-1}$ is a closed  one form in a neighborhood of zero. 

For, we start to observe that, if $u,v \in W$, then  there exists a constant $M>0$ such that 
\begin{equation}\label{eq:prima-lemma1}
	\left|\dfrac{d}{dt}\langle G_{zt} u,v\rangle_H\right|< M|z|\norm{u}_W\norm{v}_H.
	\end{equation}
By a direct integration on $[0,1]$ it follows that 
$|\langle G_z u, v\rangle-\langle G_0 u, v\rangle_H|<M|z|\norm{u}_W\norm{v}_H$ and hence $\norm{G_z-G_0}_{\Lin(W,H)}\leq M|z|$.
By letting $D_z\=G_z-G_0$, we get that $\norm{D_z}_{\Lin(W,H)}\leq M|z|$. By letting $L_z\=D_z G_0^{-1}=(G_z-G_0)G_0^{-1}$ and setting  $\delta:= (2M\norm{G_0^{-1}}_{\Lin(W,H)})^{-1}$,   we get  that 
\[
	\norm{L_z}_{\Lin(H)}= M|z| \norm{G_0^{-1}}_{\Lin(W,H)}\leq M \delta \norm{G_0^{-1}}_{\Lin(W,H)}\leq \dfrac12\qquad \textrm{ for } |z| \leq \delta.
	\]
Now, observe that $G_zG_0^{-1}=\Id +L_z$ and the path $z\mapsto I+L_z$ is also of class $\mathscr C^1$. Since $G_zG_0^{-1}$ is a bijection, we get that $z\mapsto \Id+L_z$  is a path of linear diffeomorphisms.  Thus, in particular, the path  $z\mapsto (\Id+L_z)^{-1}$ is a  $\mathscr C^1$ path of bounded operators  on $Y$.
Then, we have
	\begin{multline}\label{eq:terza-lemma1}
	\Tr\big[\dd G_zG_z^{-1}\big]\\= \Tr\big[\dd G_z\big(G_0(\Id+L_z)\big)^{-1}\big]
	=\Tr\big[G_0\dd L_zG_0^{-1}(\Id+L_z)^{-1}\big]=\Tr\big[\dd L_z(\Id+L_z)^{-1}\big]\\
	=\dd \Tr\log(\Id+L_z)
	\end{multline}
where the last equality readily follows by using  Lemma \ref{thm:fava}. So $\Tr\big[ \dd G_z\, G_z^{-1}\big]$ is exact in a neighborhood of $0$.\\
{\bf Second step.\/}  $\Tr \dd G_z G_z^{-1}$ is a closed  one form. 

First of all, we observe that by arguing as above, we get that 
\begin{equation}\label{eq:seconda-lemma1}
\norm{G_z-G_{z'}}_{\Lin(W,H)}\leq M|z-z'| \qquad \forall\, z , z' \in X
\end{equation}
and hence also that 
\[
\norm{D_z-D_{z'}}_{\Lin(W,H)}\leq M|z-z'|\qquad \forall\, z , z' \in X.
\]
We assume that $z_0\in Y$. Then we  have 
\[
(G_z-G_{z_0})G_{z_0}= (G_z-G_{z_0})G_0^{-1}G_0G_{z_0}= (L_z-L_{z_0})G_0G_{z_0}^{-1}=(L_z-L_{z_0})(1+L_{z_0})^{-1}.
\]
 So, we get that  the path $z\mapsto (G_z-G_{z_0})G_{z_0}^{-1}$ is a path in $\mathscr C^1\big(Y,\trace(H)\big)$. Since  $(G_z-G_{z_0})G_{z_0}^{-1}\in \trace (H)$, repeating verbatim the argument given in the first step, we also get that  for each $z \in Y$, there exists a neighborhood $U_z$ such that  $\Tr[\dd G_z G_z^{-1}]$ is exact. Then we get  that the form  $\Tr[\dd F_zF_z^{-1}]$ is (globally) closed on $Y$. 
 
The last claim directly follows by the previous one, once observed that, since by assumptions $C_z\in \mathscr C^1(X,\Lin(H))$ and $G_0^{-1}\in  \trace(H)$, then $z\mapsto (G_z-G_0)G_0^{-1}=C_zG_0^{-1}$ is a $\mathscr C^1$ path in $\trace(H)$. This concludes the proof. 
\qed



\section{An abstract trace formula}\label{sec:due}

The first result of this subsection, allow us to reduce the computation of the trace of an operator-valued one form to the sum of  the trace of a closed operator-valued one form  on a finite dimensional space and the trace of  an exact operator-valued one form  on a infinite dimensional space. Even if many of the results of this paragraph hold in a more general context, here we shall restrict to the self-adjoint case. 

We start by recalling that a bounded perturbation of a compact resolvent operator has compact resolvent as soon as the resolvent set is not empty. 
\begin{lem}\label{thm:compact-resolvent-perturbation}
	Let  $A_{z_0}$ be a self-adjoint operator having  compact resolvent and for each $z \in X$ we assume that $C_z\in \Lin(H)$. Then, the operator $A_z\=A_{z_0}+C_z$ has compact resolvent for each $z \in X$.
\end{lem}
\proof
This fact can be easily seen  as follows. In fact, 
\[
	(\lambda\Id - A_{z_0})- C_z= \big[(\lambda \Id -A_{z_0})-C_z(\lambda \Id- A_{z_0})^{-1}(\lambda \Id- A_{z_0})\big].
	\]
Since $A_{z_0}$ is self-adjoint, there exists $\lambda \in i\R$ such that $\Im(\lambda) > \norm{C_z}_{\Lin(H)}$. We observe that 
\[
\norm{C_z(\lambda\Id -A_{z_0})^{-1}}=\dfrac{\norm{C_z}}{\lambda}\norm{\left(\Id- \dfrac{A_{z_0}}{\lambda}\right)^{-1}}\leq\dfrac{\norm{C_z}}{\lambda}<1.
\]
Thus,  we get that
\begin{equation}\label{eq:compact-resolvent-stuff}
	(\lambda -A_{z_0}-C_z)^{-1}=\big[(\lambda \Id-A_{z_0})^{-1}(\Id-C_z(\lambda \Id -A_{z_0})^{-1})\big].
	\end{equation}
 Now, the conclusion follows by observing that the first (respectively second) operator at the right-hand side of Equation~\eqref{eq:compact-resolvent-stuff} is compact (respectively bounded). This concludes the proof. \qed
 \begin{rem}
As already observed a general bounded perturbation of a compact(trace class) resolvent operator has compact resolvent if and only if it's resolvent set is not empty!
\end{rem}
Now, since the spectrum of a compact resolvent operator  consists of isolated points of finite multiplicity, then  if $A_z$
has compact resolvent then  for each $z\in X$ there  exists $c>0$ such that 
\[
 \spec(A_{z}) \cap   S_c=\emptyset
\]
where $ S_c\=\{\lambda \in \C: |\lambda|= c\}$. 

We recall that $P_z \=-\dfrac{1}{2\pi i} \int_{\lambda=c}(A_z-\lambda I)^{-1}\dd\lambda $  is the projection to the total eigenspace corresponding to the eigenvalue of $A_z$ in the disk $|\lambda|<c$. We also have
$P_zA_z=A_zP_z$.
We set  $Q_{z}\= \Id - P_{z}$ and we define the following two operators
\begin{equation}\label{eq:M-N}
	N_{z}\=  Q_{z}+ P_{z}\, A_z\, P_z \quad \textrm{ and } \quad  M_z= P_z+ Q_z\,A_z\,  Q_z.  
	\end{equation}  
By a direct calculation we get that 
\begin{equation}\label{eq:nuova}
A_z=M_zN_z=Q_zA_zQ_z+ P_zA_zP_z.
\end{equation}
The following useful results holds.
\begin{cor}\label{thm:lemma-2}
		Under the previous assumptions, the  projection $P_z$ has finite dimensional range.
\end{cor}
\proof
 Now, since $A_z$ has compact resolvent, it follows that its spectrum consists of finitely many points (counted with their own multiplicity). Moreover, by choosing  $\lambda\in \mathbb{C}$ such that $\lambda \Id -A_z$ is invertible it follows that $\mu \to (\lambda-\mu)^{-1} $ is a one to one map from $\spec(A_z)$ to $\spec ((\lambda\Id-A_z)^{-1})$. Thus, also the set $\{z \in \C||z|<c \textrm{ and } z\in \spec(A_z)\}$ consists of  finitely many eigenvalues of $A_z$  each of them having  finite algebraic multiplicity. So the rank of $P_z$ is finite. This concludes the proof.
	\qed 
\begin{lem}\label{thm:mancante}
	 Let $X \subset \C$ be an open subset,  $A_{z_0}$ be a operator such that $A_{z_0}^{-1}\in  \trace(H)$ and  we let $A_z\=A_{z_0}+C_z$ for $C_z\in \mathscr C^1\big(X,\Lin( H)\big)$. Assume that $\spec(A_{z_0})\cap S_c=\emptyset$ for some $c>0$. Then, there exists a (convex) $W_{z_0}$ neighborhood of $z_0$ such that   $A_z=M_zN_z$ for every $z \in W_{z_0}$ where  $M_z$ and $N_z$ are given in Equation \eqref{eq:M-N}  and $N_z\in \mathscr C^1 (W_{z_0},\Lin(H))$.
\end{lem}
\proof
	Under the previous notations, for $z=z_0$, the projection defined by 
	\[
	P_z \=-\dfrac{1}{2\pi i} \int_{|\lambda|=c}(A_z-\lambda I)^{-1}\dd\lambda 
	\]
	 is well-defined and  it corresponds to the eigenprojection of $A_{z_0}$ for the part of the spectrum contained in the disk $|\lambda|<c$.  

Recall that if $T^{-1}$ is bounded invertible, then there is $\delta>0$ such that for each $\|C\|<\delta$, $T+C$ is bounded invertible and 
$\|(T+C)^{-1}-T^{-1}\|<\|C\|\|T^{-1}\|^2/(1-\|C\|\|T^{-1}\|)$.

	By the compactness of the circle $|\lambda|=c$, we get that there is $\delta >0$ such that $(A_{z_0}+C_z-\lambda I)^{-1}\in \Lin(H)$ for $|\lambda |=c$ with $\|C_z\|<\delta$. So there is a convex $W_{z_0}$ neighborhood of $z_0$ such that $P_z$ is well defined with $z\in W_{z_0}$.   Note that 
	\[
	P_z-P_{z'}=-\dfrac{1}{2\pi i}\int_{|\lambda|=c}(A_z-\lambda I)^{-1}(C_{z'}-C_z)(A_{z'}-\lambda I)^{-1}\dd \lambda .
	\]
	Since $C_z\in \mathscr C^1(X,\Lin(H))$ , we have $P_z\in \mathscr C^1(W_{z_0},\Lin(H))$.
	For $z\in W_{z_0}$, we also have $P_zA_zP_z=-\dfrac{1}{2\pi \mathrm{i}}\int_{|\lambda-z_0|=c}\lambda(A_z-\lambda\Id)^{-1}\dd\lambda$.  Using similarly method above, we have $P_zA_zP_z \in \mathscr C^1(W_{z_0},\Lin(H))$.
	Then we can conclude that $N_z=Q_z+P_zA_zP_z \in \mathscr (W_{z_0},\Lin(H))$.
	\qed
	\begin{prop}\label{thm:importante}
		Under the assumptions and notation of Lemma \ref{thm:mancante}, the following hold
		\begin{enumerate}
		\item[1.]  $M_z=M_{z_0}+D_z$with $z\mapsto D_z \in \mathscr C^1(W_{z_0}, \Lin (H))$ and $z\mapsto N_z\in \mathscr C^1(W_{z_0},\Lin (H))$;
		\item[2.] $\Tr\big[\dd M_z M_z^{-1}\big]$ is an exact  one-form on $U_{z_0}$.
		\end{enumerate}
		 Let $V_{z_0}\=\set{z\in W_{z_0}}{A_z \ \rm{ invertible}}$. Then 
		 \begin{enumerate}
					\item[3.]  $\Tr\big[ \dd A_z A_z^{-1}\big]=\Tr\big[ \dd M_zM_z^{-1}\big]+\Tr \big[\dd N_z N_z^{-1}\big]$ on $V_{z_0}$;
		\item[4.] $\Tr\big[\dd N_z  N_z^{-1}\big]$ is a closed but not exact one-form on $V_{z_0}$.
	\end{enumerate}
\end{prop}
\proof

	As in Lemma \ref{thm:mancante}, recall that 
	\begin{equation}\label{eq:projections}
	P_z=-\dfrac{1}{2\pi \mathrm{i}}\int_{|\lambda|=c}(A_z-\lambda\Id)^{-1}\dd\lambda \quad \textrm{ and }\quad
	P_zA_zP_z=-\dfrac{1}{2\pi \mathrm{i}}\int_{|\lambda|=c}\lambda(A_z-\lambda\Id)^{-1}\dd\lambda.
	\end{equation}
	By using Equation \eqref{eq:projections},  for $z$ in a sufficiently small neighborhood $W_{z_0}$ of $z_0$ it follows that the path $z\mapsto N_z=Q_z+P_zA_zP_z$ belongs to  $\mathscr{C}^1(W_{z_0}, \Lin(H))$.	By Equation \eqref{eq:nuova}, it follows that  $Q_zA_zQ_z=A_{z_0}+C_z-P_zA_zP_z$.  We set $D_z\=C_z-P_zA_zP_z+P_z-P_{z_0}+P_{z_0}A_{z_0}P_{z_0}$ and we observe that the path $z\mapsto D_z$ lies in$\mathscr{C}^1(W_{z_0},\Lin (H))$ and moreover $M_z=M_{z_0}+D_z$.

In order to prove the second claim, we observe that for all $z \in  U_{z_0}$ the operator  $M_z^{-1}$ is invertible  (since $0$ is not in the spectrum of $M_z$ as easily follows by the definition of $M$). Moreover  $\dd M_z=\dd D_z$ is locally uniformly bounded and $M_{z_0}^{-1}\in \trace(H)$ (being $M_{z_0}$ a bounded perturbation of a trace class resolvent operator  $A_0$ having not empty resolvent set).
	By invoking Lemma \ref{lm:inv_exact}, it follows that $\Tr[\dd M_zM_z^{-1}]$ is an  exact one-form on $U_{z_0}$.

To prove the third claim, we observe that, since $M_z^{-1}=N_zA_z^{-1}$, then  we have
	\begin{multline}\label{eq:1111}
	-M_z^{-1} \dd M_z M_z^{-1}=\dd(M_z^{-1})=\dd N_z(A_z^{-1})-N_z A_z^{-1}\dd A_z A_z^{-1}\\ =\dd N_z(N_z^{-1}M_z^{-1})-N_zA_z^{-1}\dd A_zN_z^{-1}M_z^{-1}.
	\end{multline}
	Since $\mathrm{dom} M_z=\mathrm{dom} A_z$ is dense in $H$, we get by the first and last members of Equation \eqref{eq:1111}, that
		\[
	M_z^{-1} \dd M_z =-\dd N_z N_z^{-1}+N_zA_z^{-1}\dd A_zN_z^{-1}
	\]
	and by using once again the commutativity property of the trace,  we also get that 
	\[
	\Tr \dd A_z A_z^{-1}=\Tr \dd M_zM_z^{-1}+\Tr \dd N_zN_z^{-1}.
	\]
In order to prove the fourth claim, we observe that  by Lemma \ref{lm:inv_exact}, the one-form $\Tr(\dd A_z A_z^{-1})$ is closed on $V_{z_0}$. Then by the claim 3. we get that 
 	$\Tr[\dd N_zN_z^{-1}]$ is also closed. This concludes the proof.	
\qed

\begin{prop} \label{lm:trace_det}
Given $\varepsilon>0$,	let  $A\in \mathscr C^1\big((-\varepsilon, \varepsilon), \GL(W,H) \big)$ be a path of invertible operators on a Hilbert space $H$ having  the same domain  $W$  and let  $P\in \mathscr C^1\big((-\varepsilon, \varepsilon), \Lin(H) \big)$ be a path of finite rank projections. We assume that 
\begin{itemize}
\item $ P_tA_t=A_tP_t$ for all $t \in (-\varepsilon, \varepsilon)$
\item $t \mapsto P_tA_tP_t \in \mathscr C^1((-\epsilon,\epsilon),\Lin(H))$
 \end{itemize}
	 Then we get that  $E_t\=P_tA_tP_t|_{\Imm P_t}$ is a linear map from $\Imm P_t$ to $\Imm P_t$ and  by setting   $N_t\=(\Id-P_t)+P_tA_tP_t$, the following equality holds:
		\[
		\Tr[\dd N_t N_t^{-1}]= \Tr[\dd (P_tA_tP_t)P_tA_t^{-1}P_t]=\dd \log\det E_t.
		\]	
\end{prop}
\proof
	Without loss of generality, we only need to prove the equality at $t=0$; thus  we  assume that $\dim \Imm P_0=n$ and let $(e_i)_{i=1}^n$ be a basis of $\Imm P_0$. 	Then for $|t|$ sufficiently small, we get  $\{P_te_1,P_t e_2,\cdots,P_t e_n\}$ is a basis of $\Imm P_t$.
	Let
	$P_tA_tP_t{e_i}=\sum_{1\le j\le n} \alpha_{ij}(t)P_t(e_j)$.
	Then we have $E_t(P_t e_j)=\sum_{1\le j\le n} \alpha_{ij}(t)P_te_j$. Furthermore
	\[
	P_0P_tA_tP_tP_0 e_i=\sum_{1\le j\le n}\alpha_{ij}(t)P_0P_tP_0 e_j.
	\]
	Let now $R_t\=P_0P_tP_0$ be the map from $\Imm P_0$ to $\Imm P_0$ and we let $R_t e_i=\sum_{1\le j\le n} \beta_{ij}(t) e_j$. We denote by $S_t$ the map from $\Imm P_0$ to $\Imm P_0$ defined by 
	 $S_t\=P_0P_tA_tP_tP_0$. 	Then we have
	\begin{equation}\label{eq:S}
	S_t e_i=\sum_{1\le j,k\le n}\alpha_{ij}(t)\beta_{jk}(t)e_k.
	\end{equation}
	By Lemma \ref{thm:lemma-3} and by Equation \eqref{eq:S}, we get 
	\[
	\Tr[\dd S_t S_0^{-1}]\big\vert_{t=0}=\dd\log\det S_t\big\vert_{t=0}=\dd\log\det E_t \big\vert_{t=0}+\dd\log\det R_t \big\vert_{t=0}.
	\]
	We note that $\dd S_t|_{t=0}S_0^{-1}=P_0\dd(P_tA_tP_t)|_{t=0}P_0(P_0A_0^{-1}P_0)$ and hence 
	\[
	\Tr[\dd S_tS_0^{-1}]\big\vert_{t=0}=\Tr\big[\dd(P_tA_tP_t)\big\vert_{t=0}P_0A_0^{-1}P_0\big].
	\]
	By invoking once again Lemma \ref{thm:lemma-3}, we also have
	\[
	\dd\log\det R_t \big\vert_{t=0}=\Tr[\dd(P_0P_tP_0)\big\vert_{t=0}P_0]=\Tr[P_0\dd P_t\big\vert_{t=0}]=\dfrac12\Tr \dd P_t^2\big\vert_{t=0}=\dfrac12\Tr \dd P_t\big\vert_{t=0}.
	\]
We note that $\Tr(P_t\dd P_t)=\Tr(P_t\dd P_t^2)=\Tr(P_t P_t\dd  P_t+P_t\dd P_tP_t)=2\Tr(P_t\dd P_t)$.
	Then we have 
	\[\Tr(\dd  P_t)=2\Tr(P_t\dd P_t)=0.
	\]
		Then we can conclude that 
	\[
	\Tr[\dd(P_tA_tP_t)\big\vert_{t=0}P_0A_0^{-1}P_0]=\dd \log \det E_t\big\vert_{t=0}\]
	By arguing as above, for each $t\in (-\epsilon,\epsilon )$, we get 
\begin{equation}\label{eq:final1}
	\Tr[\dd(P_tA_tP_t)P_tA_t^{-1}P_t]=\dd \log \det E_t.
\end{equation}
Let us consider the projection $Q_t=\Id-P_t$ and  we observe that
	\[
		\Tr[\dd N_tN_t^{-1}]=\Tr\big[\big(\dd Q_t+\dd(P_tA_tP_t)\big)(Q_t+P_tA_t^{-1}P_t)\big].
		\]
	We observe that
	\begin{equation}
	\dd(P_tA_tP_t)=\dd(P_tA_tP^2_t)= d(P_tA_tP_t)P_t+P_tA_tP_t dP_t
	\end{equation}
	and by the first and third member in the previous equation, we get that $\dd(P_tA_tP_t)Q_t= P_tA_tP_t dP_t$. By this, it follows that $\Tr(\dd(P_tA_tP_t)Q_t)=0$. 	
	
	Similarly we have 
	\[
	\Tr[\dd Q_t (P_tA_t^{-1}P_t)]=\Tr[\dd Q_t^2 (P_tA_t^{-1}P_t)]=\Tr[Q_t\dd Q_t P_t A_t^{-1}P_t+\dd Q_t Q_t(P_tA_t^{-1}P_t)]=0.
	\]
	Then we can conclude that $\Tr(\dd N_t N_t^{-1})=\Tr[\dd(P_tA_tP_t)P_tA_t^{-1}P_t]$. Summing up this last equation and Equation~\eqref{eq:final1}, we get the desired conclusion.	
\qed

For, let us consider the family of operators parametrized by the rectangle $\Omega\=[0,1]\times[-1,1]\subset \C$
\begin{equation}
		A_z\= A_t +is\, \Id\quad  \textrm{ where} \quad z\=t+is \in \Omega
		\end{equation}
		where $A_t$ is self-adjoint for every $t$. Arguing preciely as in Proposition \ref{lm:trace_det}, we assume that there is a path of finite rank projections $ P\in \mathscr C^1 ([0,1],\Lin(H))$  such that $P_tA_t=A_tP_t$ and $ t\mapsto P_tA_tP_t \in \mathscr C^1([0,1],\Lin(H))$. Let $S_z\=[P_tA_tP_t+is \Id]|_{\Imm P_t}$ and we recall that 
	there exist $n$ continuous functions that represent the repeated eigenvalues of $P_tA_tP_t|_{\Imm P_t}$ and,  up to relabel, we can assume that 
	\[
	\lambda_1(t) \leq \ldots \leq\lambda_n(t).
	\]
Thus, we have $\displaystyle\det S_z= \prod_{i=1}^n\big[\lambda_i(t)+ is\big]$. 
\begin{prop}\label{thm:lemma-4}
Under the above notation and assume that $A_0,A_1$ is bounded invertible  , $A_0^{-1}\in S_1(H)$, and $A_t=A_0+C_t$ with $C_t\in \mathscr C^1([0,1],\Lin_s(H))$.
We assume that there is c>0 such that $\pm c \notin \spec(A_t) $, for each $t\in [0,1]$. 
	Then	the decomposition $A_t=M_tN_t$ is well defined on $[0,1]$ and we have 
		\begin{equation}
		\dfrac{1}{2\pi i} \int_{\partial \Omega}\Tr \dd A_z A_z^{-1} = l-m
		\end{equation}
		where $l$ (respectively $m$) is  the number of $\lambda_i(t)$ (counted with multiplicity) crossing the real axis from negative to positive (respectively from positive to negative). 
\end{prop}
\begin{proof}
	The proof of this result is quite straightforward. 
First, we observe that  on $\Omega$  the following decomposition holds 
	\[
	A_z=M_zN_z, \textrm{ where } M_z=P_t+Q_t(A_t+is\Id)Q_t, \textrm{ and } N_z=Q_t+P_t(A_t+is \Id)P_t.
	\]
		By invoking Proposition~\ref{thm:importante}, we get 
	\begin{equation}
	\dfrac{1}{2\pi i} \int_{\partial \Omega} \Tr \dd A_z A_z^{-1} = \dfrac{1}{2\pi i} \int_{\partial \Omega} \Tr \dd N_z N_z^{-1}+ \dfrac{1}{2\pi i} \int_{\partial \Omega} \Tr \dd M_z M_z^{-1}\\= \dfrac{1}{2\pi i} \int_{\partial \Omega} \Tr \dd N_z N_z^{-1}.
	\end{equation}
	By invoking Proposition \ref{lm:trace_det}, we get 
	\[
	\dfrac{1}{2\pi i} \int_{\partial \Omega} \Tr \dd N_z N_z^{-1} =\dfrac{1}{2\pi i} \int_{\partial \Omega} \dd \log\det [P_tA_tP_t+is \Id]|_{\Imm P_t} .
	\]
	Let $S_z\=[P_tA_tP_t+is \Id]|_{\Imm P_t}$.
	There exist $n$ continuous functions that represent the repeated eigenvalues of $P_tA_tP_t$ and,  up to relabel, we can assume that 
	\[
	\lambda_1(t) \leq \ldots \leq\lambda_n(t).
	\]
Thus, we have $\displaystyle\det S_z= \prod_{i=1}^n\big[\lambda_i(t)+ is\big]$.  So, let us consider the eigenvalue $\lambda_i$ and we observe that only four cases can occur (according to the sign of $\lambda_i$ at the ends); more precisely 
\begin{equation}\label{eq:cases}
(I)\ \begin{cases}
\lambda_i(0)<0\\
\lambda_i(1)>0
\end{cases},\quad 
(II)\ \begin{cases}
\lambda_i(0)>0\\
\lambda_i(1)<0
\end{cases},\quad 
(III)\ \begin{cases}
\lambda_i(0)<0\\
\lambda_i(1)<0
\end{cases},\quad
(IV)\ \begin{cases}
\lambda_i(0)>0\\
\lambda_i(1)>0
\end{cases}.
\end{equation}
For each one of the cases appearing in Equation \eqref{eq:cases}, we construct the following four homotopies  
\begin{multline}
h_{(I)}(s,t)\=(1-s)\lambda_i(t) + s\left(t-\dfrac12\right), \quad h_{(II)}(s,t)\=(1+s)\lambda_i(t) + s\left(-t+\dfrac12\right)\\
 h_{(III)}(s,t)\=(1-s)\lambda_i(t) -s\quad h_{(IV)}(s,t)\=(1-s)\lambda_i(t) + s, \quad.
\end{multline}
Thus, there exist $h,k,l,m \in \N$ such that $\det S_z$ is a homotopic (through an admissible homotopy) to the following function
\[
\phi(t,s)\=(1+is)^k (-1+is)^h \left(t-\dfrac12+ is\right)^l\left(-t+\dfrac12+is\right)^m
\]
where $n=h+k+l+m$. Moreover the functions $s \mapsto 1+ is$ and  $s \mapsto -1+ is$ are homotopic to  the constant functions $1$ and $-1$ respectively through the homotopy $s \mapsto 1+ i\lambda s$ and $s \mapsto - 1+ i\lambda s$, for $\lambda \in [0,1]$ respectively. Thus, we have
\begin{multline}
\deg(\det S_z, R, 0)\\=
\dfrac{1}{2\pi i} \int_{\partial R} \dd \log\det S_z= \dfrac{1}{2\pi i}\int_{\partial R} \dd \log\left[\left(t-\dfrac12+ is\right)^l\left(-t+\dfrac12+is\right)^m\right]\\= \dfrac{1}{2\pi i}\int_{\partial R}  \dd \log \left(t-\dfrac12+ is\right)^l+ \dfrac{1}{2\pi i}\int_{\partial R}  \dd \log\left(-t+\dfrac12+is\right)^m\\= l-m.
\end{multline}
This concludes the proof. 
\end{proof}
\begin{rem}
The proof given in Proposition~\ref{thm:lemma-4} mainly relies on the homotopy invariance property of the winding number and it is completely differen of a similar result proved by authors in  \cite[Proposition 5.1]{MPP05} by using  Kato's Selection Theorem. 
\end{rem}


\section{Degree index and trace formulas}\label{sec:HPW-index}

This section is devoted to introduce  a new invariant  defined by means of a suspension of the complexified family of Morse-Sturm boundary value problems. The resulting boundary value problem is parameterized by points of the complex plane. This topological invariant that will be defined in terms of the  Brouwer degree of an associated determinant map.

We consider a linear second order differential operator 
\begin{equation}
	\mathscr A_0\=-\dfrac{d}{dx}\left[P(x) \dfrac{d}{dx}+ Q(x)\right]+ \trasp{Q}(x)\dfrac{d}{dx} +G(x) \qquad x \in [0,1]
	\end{equation}
with matrix coefficients $P, G \in \mathscr C^1\big([0,1],\SSS(N)\big)$, $ Q \in \mathscr C^1\big([0,1], \Mat(N, \R)\big)$ and we  assume that  $P(x)$ is non degenerate for each $x\in [0,1]$.  

Now, for every  $t \in [0,1]$, we let $C_t \in \mathscr C^1\big(I,  \SSS(N)\big)$ and we assume that $t\mapsto C_t(x)$ is continuous and $C_0(x)=0$. 

We set $C_1(x)=C(x)$ and  let us now define the second order  differential operator 
\begin{equation}
\mathscr A_t\= \mathscr A_0 + \mathscr C_t
\end{equation}
where $\mathscr C_t$ denotes the operator pointwise defined by $C_t$ as follows $(\mathscr C_t u)(x)\= C_t u(x)$ for every $x \in [0,1]$. Without further conditions,  the operators  $\mathscr A_0$ and hence $\mathscr A_t$ acts, for each $t \in [0,1]$ on the space $\mathscr C^1([0,1], \R^N)$.
\begin{defn}\label{def:bc} 
For $i=0,1$ we set  $ R_i\in \Mat(2N,\R)$ and we define the boundary operator $\mathcal R$ as 
\begin{equation}\label{eq:bc}
 \mathcal R(u):=  R_0 \begin{bmatrix}
                      P(0) u'(0)+ Q(0)\\
                      u(0)
                     \end{bmatrix}
+ R_1 \begin{bmatrix}
                     P(1) u'(1)+ Q(1)\\
                      u(1)
                     \end{bmatrix}
\end{equation}
where $'$ denotes the derivative with respect to $x$.
\end{defn}
\begin{ex} We observe that the boundary conditions given in Definition \ref{def:bc} are very general. It is worth noticing that:
\begin{itemize}
 \item  Dirichlet case corresponds to choose 
\[
 R_0:=\begin{bmatrix}
      0 & I_n\\
      0 & 0 
     \end{bmatrix}, \qquad  R_1:=\begin{bmatrix}
      0 & 0\\
      0 & I_n 
     \end{bmatrix};
   \]
\item Neumann  case corresponds to choose
\[
R_0:=\begin{bmatrix}
      I_n & 0\\
      0 & 0 
     \end{bmatrix}, \qquad  R_1:=\begin{bmatrix}
      0 & 0\\
      I_n & 0 
     \end{bmatrix};
   \] 
\item Periodic boundary conditions corresponds to choose
\[
  R_0= \begin{bmatrix}
                \Id & 0 \\
                0 & \Id
               \end{bmatrix}, \qquad  R_1= - R_0.
\]
\end{itemize}
\end{ex}

For each $t \in I$, we denote by $\mathcal A_0$ and $ \mathcal A_t$, the operator $\mathscr A_0$ and  $\mathscr A_t$ respectively  acting on the domain $\displaystyle  \mathcal D\=\Set{u \in H^2(I, \R^N)| \mathcal R u=0}$ and we consider the complexified (extension of the) operators $\mathcal A_0$ and $\mathcal A_t$ by considering both operators acting on the $\mathcal C^\infty(I, \C^N)$. With  a slight abuse of notation we will not distinguish between $\mathcal D$ and its complexification 
\begin{equation}
\mathcal D\=\Set{u \in H^2(I, \C^N)| \mathcal R u=0}.
\end{equation}
as well as between $\mathcal A_0$ and $\mathcal A_t$ and their   complex extensions.
We let 
\begin{equation}
C_z(x)\=C_t(x) + is\, \Id \quad \textrm{ for } \quad z=t+is \in \Omega
\end{equation}   
and for every $z \in  \Omega$ we define  $\mathcal A_z : \mathcal D \subset L^2([0,1], \C^N) \to L^2([0,1], \C^N)$ to be the closed unbounded operator on $\mathcal D$ pointwise defined by 
\begin{equation}\label{eq:Sturm-Liouville-equation}
\big( \mathcal A_z u\big)(x)\= \big(\mathcal A_0 u\big)(x) + \mathcal  C_z(x).
\end{equation}
We consider the Morse-Sturm equation 
\begin{equation}
	-\dfrac{d}{dx}\big[P(x) u'(x) + Q(x) u(x)\big]+ \trasp{Q}(x) u'(x)+G(x) u(x)+C_z(x) u(x)=0, \qquad x \in [0,1].
	\end{equation}
By setting   $v(x)\= P(x)u'(x)+ Q(x) u(x)$ and  $w(x)\equiv \trasp{\big(v(x), u(x)\big)}$,  Equation~\eqref{eq:Sturm-Liouville-equation} fits into the following linear Hamiltonian system 
\begin{equation}\label{eq:Hamiltonian-system}
w'(x)= J B_z(x) w(x), \qquad x \in [0,1]
\end{equation}
where $\displaystyle J\= \begin{bmatrix}
 0 & -\Id\\ \Id & 0
 \end{bmatrix}$
and finally 
\begin{equation}
B_z(x)\= \begin{bmatrix}
 P^{-1}(x) & - P^{-1}(x) Q(x)\\ - \trasp{Q}(x) P^{-1}(x) &  \trasp{Q}(x) P^{-1}(x) Q(x) - G(x)-C_z(x)
 \end{bmatrix}.
\end{equation}
Let $\psi_z:I\to \Sp(2N)$ be the fundamental solution of the Hamiltonian system given in Equation~\eqref{eq:Hamiltonian-system} and we let 
\begin{equation}
R_z\= R_0 +R_1 \psi_z(1), \qquad z \in \Omega.
\end{equation}

\begin{lem}\label{lem:well-posedness-rho}
The following statements are equivalent
\begin{enumerate}
\item $\ker \mathcal A_z \neq \{0\}$
\item $\det R_z =0$
\end{enumerate}
Setting  $\mathcal Z\=\Set{z \in \Omega|\det R_z =0}$ we have that 
\[
\mathcal Z \subset \R
\]
\end{lem}
\proof 
$(\Rightarrow)$ Let $u \in \ker \mathcal A_z$ and let $w$ be a solution of Equation~\eqref{eq:Hamiltonian-system}. Thus $w(1)= \psi_z(1) w(0)$  and $R_z w(0)= R_0 w(0) + R_1 \psi_z(1) w(0)= R_0 w(0)+ R_1 w(1)$ and by this  we get that $R_z w(0)=0$.  So now,  we assume by contradiction that $\det R_z \neq \{0\}$ then we have $w(0)=0$ which in particularly means that $u(0)= u'(0)=0$ (being $P(x)$ non-degenerate for every $x \in [0,1]$). Thus by the existence and uniqueness theorem for first order linear differential equations, we get that $w\equiv 0$.  Thus $\ker \mathcal A_z =\{0\}$ which is a contradiction. This concludes the proof of the only if part.\\
$(\Leftarrow)$ To prove the if part we start to observe that since  $\det R_z = \{0\}$, then 
there exists a non-trivial $w_0 \in \ker R_z$. We set $w(x)= \psi_z(x) w_0$ and we observe that $w$ is a solution of Equation \eqref{eq:Hamiltonian-system} and satisfies the boundary condition. It is immediate to check that, if $\mathcal P_z: \R^{2N} \to \R^N \oplus \R^N$ is the projection onto the second component, then $u(x)=\mathcal P_z w(x) $ lies in $\ker A_z$. 

In order to prove the last statement, it is enough to observe that the spectrum of a  self-adjoint operator is real; thus $\ker A_z \neq \{0\}$ can occur only at $z=t+is$ with $s=0$. This concludes the proof. 
\qed
\begin{defn}\label{def:rho}
Under the notation above, we define the {\em determinant map\/} $\rho$ as follows:
\[
 \rho: \Omega \ni z  \longmapsto \rho(z) :=  \det\, \mathcal R_z\in  \C.
\]
\end{defn}
We assume that $\mathcal A_0$ and $\mathcal A_1$ are non-degenerate  meaning that $\ker \mathcal A_0= \ker \mathcal A_1 = \{0\}$. In shorthand notation we simply refer to $t\mapsto \mathcal A_t$ or simply $\mathcal A$ as {\em admissible\/}. Thus in particular, by Lemma \ref{lem:well-posedness-rho}, we get that $0 \notin \rho\big(\partial \Omega\big)$. 
\begin{defn}\label{def:PW-index}
Under the previous notation, if $0 \notin \rho\big(\partial \Omega\big)$, we define the {\em degree index\/} associated to the pair $(\psi, R)$ as the integer  $\icon(\psi, R)$  defined by 
\begin{equation}\label{eq:HPW-index}
\icon (\psi, R) \= \deg ( \rho , \Omega , 0).
\end{equation}
We term {\em spectral index\/} of $\mathcal A$  the integer  $\ispec(\mathcal A)$  given by 
\begin{equation}\label{eq:spectral-HPW-index}
\ispec (\mathcal A) \= \spfl(\mathcal A, [0,1]).
\end{equation}
\end{defn}
\begin{rem}\label{rem:utile}
It is worth noticing that the Brouwer degree of a map $\rho$ on the open set $\Omega$ is related to the winding number as follows 
\begin{equation}
\deg(\rho, \Omega, 0)= \dfrac{1}{2\pi i} \int_{\partial \Omega} \dd\log \rho_z . 
\end{equation}
\end{rem}
Now, to the family $\mathcal A_z$ defined at Equation~\eqref{eq:Sturm-Liouville-equation},   we associate the operator valued one-form defined by $\Theta_z\=\dd \mathcal A_z \mathcal A_z^{-1}$. More explicitly 
\begin{equation}
\Theta_z\=\dd \mathcal A_z \mathcal A_z^{-1}= \partial_t \mathcal C_z \mathcal A_z^{-1}\, \dd t + \partial_s \mathcal C_z \mathcal A_z^{-1}\, \dd s= 
\partial_t \mathcal C_t \mathcal A_z^{-1}\, \dd t + i \mathcal A_z^{-1}\, \dd s 
\end{equation}
defined on the subset $\Set{z \in \Omega| \mathcal A_z \textrm{ has a bounded inverse}}$.

By invoking Lemma \ref{lem:well-posedness-rho} and \cite[Theorem 3.1, pag. 295-296]{GGK90}, $\mathcal A_z$ has a bounded inverse on $\partial \Omega$. Moreover $\mathcal A_z^{-1}$ is a Hilbert-Schmidt operator. In fact, we have
\begin{equation}
\mathcal A_z^{-1}\, u(x)= \int_0^1 K_z(x,y) u(y)\,dy 
\end{equation}
where the Green's kernel is given by 
	\begin{equation}
	K_z(x,y)= -\trasp{C}\widetilde K_z(x,y) D
	\end{equation}
	where $C=\trasp{[0, \Id]}, D=\trasp{[ \Id, 0]}$ and finally 
\begin{equation}
\widetilde K_z(x,y)\= 
\begin{cases}
\psi_z(x) \mathcal P_z \psi_z(y)^{-1} & 0 \leq x < y \leq 1\\
\psi_z(x)(\Id-\mathcal P_z) \psi_z^{-1}(y) & 0 \leq y < x \leq 1
\end{cases}
\end{equation}
for  $\mathcal P_z\=[R_0 + R_1 \psi_z(1)]^{-1} R_1 \psi_z(1)= R_z^{-1}R_1 \psi_z(1)$. \begin{rem}
 It is worth to note that $\widetilde K$ is not continuous on the diagonal  whilst $K$ is. In fact, it is immediate to observe that the multiplication of $K$ by $C$ on the left and $D$ on the right, actually corresponds to consider the lower left block entry in the block decomposition of $\widetilde K_z$ (the identity is invisible in that block).  
 \end{rem}
\begin{lem}\label{thm:trace-class-form}
The form $\Theta_z$ is a trace class operator-valued one form. 
\end{lem}
\begin{proof}
Since the domain $\mathcal D$ of $\mathcal A_z$ is a dense ($z$-independent) subspace of $\subset H^2\big([a,b], \C^N\big)$  and by taking into account Lemma \ref{thm:Sobolev-inclusion-trace-class}, we get that the operator $\mathcal A_z^{-1}$ is a trace class operator on $L^2$. In fact $\mathcal A_z^{-1}$  is the composition of a bounded operator and of a trace class operator and hence it is in the  trace class operator (being, in fact, trace class operators  an ideal of the ring of all compact operators). Thus  $\Theta_z$ is a trace class operator-valued one form. This concludes the proof. 
 \end{proof}
 The next proposition is crucial in order to establish the relation between the trace of $\Theta_z$ and the Brouwer  degree of the determinant map. 
 \begin{prop}\label{thm:traccia=grado}
 Under the above notation, we get 
 \begin{equation}
 \dfrac{1}{2\pi i}\Tr \int_{\partial \Omega} \Theta_z= \deg (\rho, \Omega, 0).
 \end{equation}
 \end{prop}
\proof 
It is well-known that the trace of an integral operator belonging to the trace class can be computed by integrating the trace of its kernel. (Cfr. \cite{LT98} for further details).  Thus, by a direct calculation, we get: 
\begin{equation}
	\Tr \Theta_z = \int_0^1 \Tr\big[\dd C_z(x)\, K_z(x,x)\big]\, \dd x= -\int_0^1 \Tr\big[J \dd B_z(x)\widetilde K_z(x,x)\big]\, \dd x
	\end{equation}
where the last equality follows by taking into account that 
\begin{equation}
	\dd C_z(x)= [\Id,0]J\, \dd B_z(x) \trasp{[0,\Id]} \quad \textrm{ and } \quad 
	K_z(x,y)= -\trasp{C}\widetilde K_z(x,y) D
	\end{equation}
as well as the commutativity of trace.
On the other hand, 
\begin{multline}\label{eq:1-prop4.4}
	\Tr\Big[J\dd B_z(x)\psi_z(x)(\Id-\mathcal P_z) \psi_z^{-1}(x)\Big]= \\
	\Tr\Big[J\dd B_z(x)\psi_z(x)(-\mathcal P_z) \psi_z^{-1}(x)\Big]+\Tr\Big[ J\dd B_z(x)\Big]=\\ 
	\Tr\Big[J\dd B_z(x)\psi_z(x)(-\mathcal P_z) \psi_z^{-1}(x)\Big].
	\end{multline}
Moreover 
	\begin{multline}\label{eq:2-prop4.4}
\Tr\Big[J\dd B_z(x)\psi_z(x)(-\mathcal P_z) \psi_z^{-1}(x)\Big]=\\
 -\Tr\Big[J\dd B_z(x) \psi_z(x)\mathcal P_z\psi_z^{-1}(x)\Big]= -\Tr\Big[\dd\psi'_z(x)\mathcal P_z \psi_z^{-1}(x)-JB_z(x) \dd\psi_z(x)\mathcal P_z\psi_z^{-1}(x)\Big]\\=-
	\Tr\Big[\dd\psi'_z(x)\mathcal P_z \psi_z^{-1}(x)- \psi_z'(x)\psi_z^{-1}(x) \dd\psi_z(x)\mathcal P_z\psi_z^{-1}(x)\Big]\\=
	-\Tr\dfrac{\dd}{\dd x}\Big[\dd \psi_z(x) \mathcal P_z\psi_z^{-1}(x)\Big].
	\end{multline}
Putting together Equation~\eqref{eq:1-prop4.4} and Equation~\eqref{eq:2-prop4.4}, we finally get 
\begin{multline}
\Tr \Theta_z =  \int_0^1 \Tr\big[\dd C_z(x)\, K_z(x,x)\big]\, \dd x\\=- \int_0^1 \Tr\big[J \dd B_z(x)\widetilde K_z(x,x)\big]\, \dd x\\=\int_0^1
\dfrac{\dd}{\dd x}\Tr\Big[\dd \psi_z(x) \mathcal P_z\psi_z^{-1}(x)\Big]\, \dd x= 
\Tr\Big[\dd \psi_z(1) \mathcal P_z\psi_z^{-1}(1)\Big]\\= \Tr\Big[\dd \psi_z(1) R_z^{-1} R_1\psi_z(1)\psi_z^{-1}(1)\Big]= \Tr\Big[\dd \psi_z(1) R_z^{-1} R_1\Big].
\end{multline}
By Jacobi's formula, we get also that
\begin{equation}
\dd\log \det R_z = \Tr \Big[R_z^{-1}\dd R_z\Big]
= \Tr\Big[R_z^{-1}R_1\dd\psi_z(1)\Big]= \Tr\Big[\dd\psi_z(1)R_z^{-1}R_1\Big]
\end{equation}
and thus 
\begin{equation}
\Tr \Theta_z= \dd \log \det R_z.
\end{equation}
Integrating over $\partial \Omega$, we than conclude  that 
\begin{equation}
\dfrac{1}{2\pi i} \int_{\partial \Omega} \Theta_z= \dfrac{1}{2\pi i}\int_{\partial \Omega} \dd\log \det R_z= \deg(\rho, \Omega, 0).
\end{equation}
This concludes the proof. 
\qed 
\begin{rem}
We observe that this computation is not affected if we replace the term $\psi_z(x)(\Id-\mathcal P_z) \psi_z^{-1}(y)$ entering in 
		$\widetilde K_z(x,y)$ with  $-\psi_z(x)\mathcal P_z\psi_z^{-1}(y)$.	
		\end{rem}
Now we are in position to state and to prove the main result of this section. This result establish an equality between the spectral flow  for a path of (unbouded) self-adjoint trace class resolvent operators and the degree index.

\begin{rem}
		Recall that if an operator has compact (trace class) resolvent, it is a Fredholm operator. The spectral flow is defined for a path of self-adjoint Fredholm operators, so it is also well defined for a path of self-adjoint  operators having compact resolvent.
\end{rem}

\begin{thm}\label{thm:main}
		Under the above notation and if $\mathcal A$ is admissible, then \ we have 
		\begin{equation}
		\ispec(\mathcal A)= \icon(\psi, R).
		\end{equation}
	\end{thm}
\proof 
		We start to recall that pointwise, the path $\mathcal A$ is defined by $\mathcal A_t= \mathcal A_0+ \mathcal C_t$. 
		
		First of all, without leading in generalities, we  assume that $\Set{t \in (0,1)|\ker \mathcal A_t\neq \{0\}}$ has  has finite cardinality. If not,  by \cite[Theorem 4.22]{RS95},  for almost every $\delta\in \mathbb{R}$, $\mathcal A_t-\delta \Id$ has only regular crossings. (We refer to  Section \ref{sec:sf}  and references therein, for the basic definitions and properties on the spectral flow). Since regular crossings are isolated, then on a compact interval are in a finite number.
				We observe also that  
		\begin{equation}
		\lim_{\delta\to 0}\dfrac{1}{2\pi i}\Tr \int_{\partial \Omega} 
		\dd\mathcal A_z (\mathcal A_z-\delta \Id)^{-1}= \dfrac{1}{2\pi i}\Tr \int_{\partial \Omega} 
		\dd\mathcal A_z (\mathcal A_z)^{-1},
		\end{equation}
		and by taking into account the homotopy invariance of the spectral flow, we get that if $\delta$ is sufficiently small then $\spfl(\mathcal A_t, t \in [0,1]) =\spfl(\mathcal A_t-\delta\Id, t \in [0,1]) $.  So, we only need to prove the theorem in the case in which  $\Set{t\in [0,1]|\ker  \mathcal A_t\neq \{0\}}=\Set{t_1,\ldots, t_k}$. By using Lemma~\ref{lem:well-posedness-rho}, we recall that 	\[
	\Set{(t_1,0),\ldots, (t_k,0)}= \Set{ z \in \Omega|\ker  \mathcal A_z\neq \{0\} }.
	\]		
By Lemma \ref{thm:mancante}, for each $1\le i \le k$, there is a rectangular neighborhood  $\Omega_i$  of $t_i$ such that the decomposition $A_z=M_zN_z$ is well defined on $\Omega_i$. 
		By Lemma \ref{thm:trace-class-form} $\dd\mathcal A_z\mathcal A_z^{-1}$ is a trace class operator-valued one form and by Proposition~\ref{thm:importante} its trace is  closed; thus  we have
		\begin{equation}
		\dfrac{1}{2\pi i}\Tr \int_{\partial \Omega} 
		\dd\mathcal A_z \mathcal A_z^{-1}= \sum_{i=1}^k \dfrac{1}{2\pi i}\Tr \int_{\partial \Omega_i} 
		\dd\mathcal A_z \mathcal A_z^{-1}.
		\end{equation}
		By invoking Proposition~\ref{thm:lemma-4} we infer that 
		\begin{equation}
		\dfrac{1}{2\pi i}\Tr \int_{\partial \Omega_i} 
		\dd\mathcal A_z \mathcal A_z^{-1}= l-m
		\end{equation}
		where $l$ (resp. $m$) is  the number of eigenvalues (counted with multiplicity)  which cross the real axis with positive (resp. negative) derivative which is nothing but the spectral flow of the path $t \mapsto \mathcal A_t$ on a small neighborhood of $t_i$ .
		So we have 
		\begin{equation}
		\sum_{i=1}^k \dfrac{1}{2\pi i}\Tr \int_{\partial \Omega_i} 
		\dd\mathcal A_z \mathcal A_z^{-1}=\spfl({\mathcal A_t, t \in [0, 1]}).
		\end{equation}	
		Thus we get  
		\begin{equation}
		\spfl(\mathcal A_t, t \in[0,1]) = \dfrac{1}{2\pi i}\Tr \int_{\partial \Omega} \Theta_z.
		\end{equation}
		By taking into account Definition \ref{def:PW-index}, Remark \ref{rem:utile} and Proposition \ref {thm:traccia=grado}, we get the thesis. 
		\qed
	
In many interesting applications often occurs that $P(x)$ (the principal symbol fo the Morse-Sturm operator) is actually represented by a  positive definite (symmetric) matrix. In this case, in fact,  the operator $\mathcal A_t$ has a well-defined Morse index for each $t\in [0,1]$. Denoting the Morse index of $\mathcal A_t$ by $m^-(\mathcal A_t)$, then we get the following immediate consequence. 
	\begin{cor}
	Under the assumption of Theorem~\ref{thm:main} and assuming that  $P(x)$ is positive definite for every $x \in [0,1]$, then we get 
		\[
		m^-(\mathcal A_0)-m^-(\mathcal A_1)= \iota_{PW}(\psi,R) .
		\]
\end{cor}
\proof 
The proof readily follows by Theorem \ref{thm:main} and Equation~\eqref{eq:miserve}. 
\qed


\subsection{Hill's determinant formula}

In this paragraph we establish the relation between the trace formula proved in Theorem \ref{thm:main} and the classical Hill's determinant formula. For, let us consider the eigenvalues problem for the standard Morse-Sturm system, given by 
\begin{equation}\label{eq:second-order}
-\dfrac{d}{dx}\left[P(x) \dfrac{du}{dx}+ Q(x)u\right]+ \trasp{Q}(x)\dfrac{du}{dx}+ \big(G(x) + t G_1(x) \big)u=0 ,\qquad x \in [0,1]
\end{equation}
for $\lambda \in \R$, $Q\in \mathscr C^1\big([0,1], \Mat(N,\R)\big)$, $P,R, R_1 \in \mathscr C^1\big([0,1], \SSS(N)\big)$ and $P(x)$ is non-degenerate for every $x \in [0,1]$. By making use of the Legendre transformation, the linear system given in Equation \eqref{eq:second-order} reduces to the following Hamiltonian system 
\begin{equation}\label{eq:Ham-eigenvalue}
z'(x)= J B_t (x) z(x). 
\end{equation}
Let  $\Lambda_0 $ and $\Lambda_1$ be two Lagrangian subspaces of $(\R^{2n}, \omega_0)$.  Let  $Z_0$ and $Z_1$ be  two $2n\times n$ matrix such that the spans of their column vectors  are $\Lambda_0$ and $\Lambda_1$ respectively and we call them the Lagrangian frames of $\Lambda_0$ and $\Lambda_1$ .
Let us denote by $\psi_\lambda$ the fundamental solution of Equation \eqref{eq:Ham-eigenvalue}.

It is well-known that  the operator Morse-Sturm operator 
\[
\displaystyle \mathcal A\=-\dfrac{d}{dx}\left[P(x) \dfrac{d}{dx}+ Q(x)\right]+ \trasp{Q}(x)\dfrac{d}{dx}+ G(x)
\]
 is self-adjoint in $L^2(I, \R^n)$  with dense domain
\[
\mathcal D(\Lambda_0, \Lambda_1) =\Set{u \in H^2([0,1], \R^n)| z(0) \in \Lambda_0, z(1) \in \Lambda_1}.
\]
We assume $\mathcal A$ is non-degenerate (meaning that $\ker \mathcal A=\{0\}$). It is clear that $\lambda$ is a non-zero eigenvalue of $\mathcal A$ if and only if $-\dfrac{1}{\lambda}$ is an eigenvalue of $\mathcal G_1\mathcal A^{-1}$, where $\mathcal G_1$ is the operator on $L^2$ pointwise induced by $G_1$. The following result holds. See \cite[Theorem 1.1]{HW16})
\begin{thm}(Hu \& Wang, 2016)\label{thm:HW2016} If $\mathcal A$ is non-degenerate, then we have 
\begin{equation}\label{eq:Hill-formula}
\prod_{j}\Big(1- \lambda_j^{-1}\Big)= \det\big(\psi_1(1) Z_0, Z_1\big)\cdot \det\big(\psi_0(1) Z_0, Z_1\big)^{-1}
\end{equation}
where the left-hand side  is an infinite product which runs on all the eigenvalues $\lambda_j$ counted with multiplicity.
\end{thm}
We observe that the left-hand side  in Equation~\eqref{eq:Hill-formula} is the Fredholm determinant of $\det(\Id + \mathcal G_1 \mathcal A^{-1})$. 
Since $Z_0$ is a basis of $V_0$, we have $Z_0^*JZ_0=0$.
Similarly we have $Z_1^*JZ_1=0$.
Then $\begin{pmatrix}
	(JZ_0)^*\\0
\end{pmatrix} $and $\begin{pmatrix}
0\\(JZ_1)^*
\end{pmatrix}$ are $R_0$ and $R_1$ of the boundary condition $z_0\in \Lambda_0 , z_1\in \Lambda_1$.
Choose matrix $P,Q$ such that $\begin{pmatrix}
R_0&R_1\\P&Q
\end{pmatrix}$ is invertible.

Then we have 
$\begin{pmatrix}
R_0&R_1\\P&Q
\end{pmatrix}\begin{pmatrix}
-Z_0&0\\0&Z_1
\end{pmatrix}=\begin{pmatrix}
0&0\\-PZ_0&QZ_0
\end{pmatrix}$.
Then the rank of $\begin{pmatrix}
0&0\\-PZ_0&QZ_0
\end{pmatrix}$ is $2n$ ,so $\begin{pmatrix}-PZ_0& QZ_0\end{pmatrix}$ is invertible.
Let $A=\begin{pmatrix}
	R_0&R_1\\P&Q
\end{pmatrix}$.

Note that
\begin{multline}
\det(\psi Z_0,Z_1)\det A =\det(\begin{pmatrix}
R_0&R_1\\P&Q
\end{pmatrix}\begin{pmatrix}Id&-Z_0&0\\\psi&0&Z_1\end{pmatrix})\\
=\det\begin{pmatrix}
R_0+R_1\psi &0&0\\
P+Q\psi &-PZ_0&QZ_1
\end{pmatrix}=\det(R_0+R_1\psi )\det(\begin{pmatrix}
-PZ_0&QZ_1
\end{pmatrix}).
\end{multline}
It follows that $\rho(z)=\det(\psi_z(1))\det (A)\det(\begin{pmatrix}
-PZ_0&QZ_1
\end{pmatrix})^{-1}$ .
Then we have
 \begin{equation}
 \prod_{j}\Big(1- \lambda_j^{-1}\Big)= \rho(1)\rho(0)^{-1}.
 \end{equation}

\begin{rem}\label{rem:hill}
In fact, in the proof of \ref{thm:HW2016}, they proved that
\[
\det(1+z\mathcal G_1 \mathcal A^{-1})=\rho(z)\rho(0)^{-1}, z\in \mathbb{C}.
\]
\end{rem}
\begin{cor}
		Let	$ f(z)=\det(1+(t\mathcal G_1+i s\Id) \mathcal A^{-1})$. 
		Assume that  $\mathcal A$ and $\mathcal A+\mathcal G_1$ are both invertible. Let $\Omega=[0,1]\times [-1,1] $. We have
		$\deg(f,\Omega,0)=\spfl(\mathcal A+t\mathcal G_1,t\in [0,1])$.
	\end{cor}
	\proof
	By Proposition \ref{thm:traccia=grado} and Theorem \ref{thm:main}, we only need to show that $f(z)=\rho(z)\rho(0)^{-1}$.
	Note that if $\mathcal A+\mathcal G_1$ is invertible, we have
	\[1+(t\mathcal G_1+i s\Id) \mathcal A^{-1}=(1+is\Id (\mathcal A+t\mathcal G_1)^{-1})(1+t\mathcal G_1\mathcal A^{-1}).
	\]
	With the same perturbation method in theorem \ref{thm:main}, we can assume that $(A+tG_1)$ is invertible except for a finite number of $t$.
	By the multiplication property of Fredholm determinant and remark \ref{rem:hill},
	we have
	$f(z)=\rho(z)\rho(t)^{-1}\rho(t)\rho(0)^{-1}=\rho(z)\rho(0)^{-1}$
	on a dense subset of $\Omega$.
	Then the equation holds for all $z\in \Omega$ by the continuity of Fredholm determinant.
	\qed



 \section{Parity of the degree index and instability of periodic orbits}\label{sec:variational}

In this section, by using  Theorem~\ref{thm:main}, a Morse-type index theorem together with a characterization of the linear instability for a periodic orbit of a Hamiltonian system, we establish a sufficient condition in terms of the degree index for detecting the linear instability. 

Let $T\R^n\cong \R^n \oplus \R^n$ be the tangent space of $\R^n$ endowed with coordinates $(q,v)$. Given  $T>0$ and the Lagrangian function  $L \in \mathscr C^2([0,T]\times T\R^n, \R)$, we assume that the following two assumptions hold
\begin{itemize}
\item[{\bf (L1)\/}]  $L$ is non-degenerate with respect to $v$,  meaning that  the quadratic form  
\[
\langle \nabla_{vv} L(t,q,v) w, w\rangle \quad \textrm{ is non-degenerate } \quad \forall\,  t\in [0,T],\ \   \forall\, (q,v)\in T\R^n
\]
\item[{\bf (L2)\/}] $L$ is {\em exactly quadratic\/} in the velocities $v$ meaning that the function $L(t,q,v)$ is a polynomial of degree at most $2$ with respect to $v$.
 \end{itemize}
Under the assumption (L1)  the Legendre transform  
\begin{equation}
\mathscr L_L:[0,T] \times T\R^n \to [0,T] \times T^*\R^n, \qquad (t,q,v) \mapsto \big(t,q,D_v L(t,q,v)\big)
\end{equation}
is a $\mathscr C^1$ (local) diffeomorphism. The Fenchel transform of $L$ is the non autonomous Hamiltonian on $T^*\R^n$
\begin{equation}
H(t,q,p)\=\max_{v \in T_q M}\big(p[v]-L(t,q,v)\big)= p[v(t,q,p)] -L\big(t,q,v(t,q,p) \big)
\end{equation}
where $\big(t,q,v(t,q,p)\big)=\mathscr L_L^{-1}(t,q,p)$. 
\begin{rem}
	The assumption (L2) is in order to guarantee that the action functional is twice Frechét differentiable. It is well-known, in fact, that the smoothness assumption on the Lagrangian is in general not enough. The growth condition required in (L2) is related to the regularity of the Nemitski operators. For further details we refer to \cite{PWY19}  and references therein. 
\end{rem}
We denote by $H\=W^{1,2}([0,T], \R^n)$ be the  space of paths having Sobolev regularity $W^{1,2}$ and we define the Lagrangian action functional $\mathbb A: H\to \R$ as follows
\begin{equation}\label{eq:action}
	A(x)=\int_0^T L\big(t, x(t), x'(t)\big)\, dt.
\end{equation}
Let $Z \subset \R^n \oplus \R^n$ be a linear subspace and let us consider the linear subspace 
\[
H_Z\=\Set{x \in H| \big(x(0), x(T)\big) \in Z}.
\]
\begin{note}
In what follows we shall denote by $A_Z$ the restriction of the action $A$ onto $H_Z$; thus in symbols we have $A_Z\= A\big\vert_{H_Z}$.
\end{note}
It is well-know that critical point of the functional $A$ on the $H_Z$ are weak (in the Sobolev sense) solutions of the following boundary value problem
\begin{equation}\label{eq:bvp}
\begin{cases}
\dfrac{d}{dt}\partial_v L\big(t, x(t), x'(t)\big)= \partial_q L\big(t, x(t), x'(t)\big), \qquad t \in [0,T]\\
\big(x(0), x(T)\big) \in Z, \quad \Big(\partial_v L\big(0, x(0), x'(0)\big), -\partial_v L\big(T, x(T), x'(T)\big)\Big)\in Z^\perp
\end{cases}
\end{equation}
where $Z^\perp$ denotes the orthogonal complement of $Z$ in $T^*\R^n$  
and up to standard elliptic regularity arguments, classical (i.e. smooth) solutions. 
\begin{rem}
We observe, in fact, that there is an identification of $Z\times Z^\perp$ and the  conormal subspace of $Z$, namely $N^*(Z)$ in $T^*\R^n$. For further details, we refer the interested reader to \cite{APS08}.	
\end{rem}
We assume that $x \in H_Z$ is a classical solution of the boundary value problem given in Equation~\eqref{eq:bvp}.  We observe that,  by assumption (L2) the functional $A$ is twice Fréchet differentiable on $H$. Being the evaluation map from  $H$ onto $H_Z$ a smooth submersion (cfr. \cite{Kli83} for further details), also the restriction $A_Z$ is twice  Fréchet differentiable and by this  we get that $d^2A_Z(x)$ coincides with $D^2A_Z(x)$.

By computing the second variation of $A_Z$  at $x$, we get the  {\em index form\/} $I$ defined by
\begin{equation}
d^2A_Z(x)[\xi, \eta]\=I[\xi,\eta] = \int_0^T \big[\langle P(t)\xi'+Q(t) \xi, \eta'\rangle+ \langle \trasp{Q}(t)\xi', \eta \rangle + \langle R(t) \xi, \eta\rangle \big] \, dt, \qquad \forall\, \xi,\eta \in H_Z
\end{equation}
\begin{multline}
\textrm{ where }\  P(t)\=\partial_{vv}L\big(t,x(t),x'(t)\big), \quad  Q(t)\=\partial_{qv}L\big(t,x(t),x'(t)\big)\\ \textrm{ and finally }  R(t)\=\partial_{qq}L\big(t,x(t),x'(t)\big).
\end{multline}
Now, by linearizing the ODE  given in Equation \eqref{eq:bvp} at $x$, we finally get  the (linear) Morse-Sturm boundary value problem defined as follows
\begin{equation}\label{eq:MS-system}
\begin{cases}
	-\dfrac{d}{dt}\big[P(t)u'+ Q(t) u\big] + \trasp{Q}(t) u'+R(t)u=0, \qquad t \in [0,T]\\
	\big(u(0), u(T)\big) \in Z, \quad \Big(Pu'(0)+Q(0)u(0),-\big[P(T)u'(T)+Q(T) u(T)\big]\Big)\in Z^\perp
	\end{cases}
\end{equation}
We observe that $u$ is a weak (in the Sobolev sense)  solution of the boundary value problem given in  Equation \eqref{eq:MS-system} if and only if $u \in \ker I$. Moreover,  by elliptic bootstrap it follows that $u$ is a smooth solution.  

Let us now consider the {\em standard symplectic space\/} $T^*\R^n\cong \R^n \oplus \R^n$ endowed with the {\em canonical symplectic form\/}
\begin{equation}\label{eq:standard-form}
	\omega_0\big((p_1,q_1),(p_2,q_2)\big)\= \langle p_1, q_2 \rangle - \langle q_1, p_2 \rangle. 
\end{equation}
Denoting by $J_0$ the {\em (standard) complex structure\/}  namely the automorphism $J_0:T^*\R^n\to T^*\R^n$  defined  by  $  J_0(p,q)=(-q, p) $   and whose associated matrix is  given by 
\begin{equation}\label{eq:J0}
J_0= \begin{pmatrix}
  0 & -\Id\\
  \Id & 0  
 \end{pmatrix} 
\end{equation}
it immediately follows that $\omega_0 (z_1,z_2)\= \langle J_0 z_1, z_2\rangle$ for all $z_1,z_2 \in T^*\R^n $.
\begin{note}
In what follows, $T^*\R^n$ is endowed with a coordinate system  $z=(p,q)$, where $p=(p_1, \dots, p_n)\in \R^n$ and $q=(q_1, \dots, q_n)\in \R^n$. we shall refer to $q$ as {\em configuration variables\/} and to $p$ as the {\em momentum variables\/}. 
\end{note}
By setting $z(t)\= \trasp{\big(P(t)u'(t)+Q(t)u(t), u(t)\big)}$,  the Morse-Sturm equation reduces to the following (first order) Hamiltonian system in the standard symplectic space
\begin{multline}\label{eq:hamsys-finale}
z'(t)=J_0 B(t)\, z(t), \qquad t \in [0,T] \quad \textrm{ where }\\
B(t)\=\begin{bmatrix}
	P^{-1}(t) & -P^{-1}(t) Q(t)\\
	-\trasp{Q}(t)P^{-1}(t) & \trasp{Q}(t) P^{-1}(t) Q(t) -R(t)
\end{bmatrix}
\end{multline}
We now define the {\em double standard symplectic space\/} $(\R^{2n}\oplus \R^{2n}, -\omega_0 \oplus \omega_0)$ and we introduce the matrix  $\widetilde J_0\=\diag(-J_0,J_0)$ where $\diag(*,*)$ denotes the $2 \times 2$ diagonal block matrix. In this way,  the subspace $L_Z$ given  by 
\begin{equation}\label{eq:utile-p-focali}
L_Z\= \widetilde J_0(Z^\perp\oplus Z)
\end{equation}
is thus Lagrangian.  
\begin{note}
The following notation will be used throughout the paper. 
If $x$ is a solution of \eqref{eq:bvp} we denote by  $z_x$ the corresponding function defined by 
\begin{equation}\label{eq:z-x}
\big(t, z_x(t)\big)= \mathscr L_L\big(t,x(t),x'(t)\big).
\end{equation} 
\end{note}

Let us now consider the path $s \mapsto \mathcal L_s$ of unbounded Hamiltonian operators that are self-adjoint in $L^2$ and defined on the domain $D(T,L)$:
\[
\mathcal L_s\=-J_0 \dfrac{d}{dt}- B_s(t)
\]
where $s \mapsto B_s(t)$ is a $\mathscr C^1$ path  of symmetric matrices such that $B_0(t)=0_{2n}$ and $B_1(t)=B(t)$.
\begin{lem}{\bf (Spectral flow formula)\/}\label{thm:spfl}
Under the above notation, the following equality holds
\begin{equation}
	-\spfl\left(\mathcal L_s, s \in [0,1]\right)= \iCLM(L, \Graph \psi(t), t \in [0,T])
\end{equation}
where $\psi$ denotes the solution of 
\[
\begin{cases}
	\dfrac{d}{dt}\psi(t)= J_0\, B(t) \psi(t), \qquad t \in [0,T]\\
	\psi(0)=\Id_{2n}.
\end{cases}
\]
\end{lem}
\proof 
 For the proof of this result, we refer the interested reader to \cite[Theorem 2.5, Equation (2.7) \& Equation (2.19)]{HS09}.
\qed
  \begin{rem}
   The basic idea behind the proof of Proposition \ref{thm:spfl} is to perturb the path $s\mapsto \mathcal L_s$ in order to get regular crossing and without changing the spectral flow (as consequence of the fixed endpoints homotopy invariance). Once this has been done, for concluding, it is enough to prove that the local contribution at each crossing instant to the spectral flow is the opposite of the local contribution to the Maslov index. This can be achieved by comparing the crossing forms as  in \cite[Lemma 2.4]{HS09} and to prove that the crossing instants for the path $s\mapsto \mathcal L_s$  are the same as the crossing instants of the path   $s\mapsto \mathcal \Graph \psi_s$ and at each crossing $s_0$ the kernel dimension of the operator $\mathcal L_{s_0}$ is equal to the $\dim(L\cap \Graph \psi_{s_0})$. The conclusion follows once again by using the homotopy properties of the $\iCLM$-index and the spectral flow. 
 \end{rem}

\begin{defn}
Let $x$ be a critical point of $A$. We denote by $\iMorse{ Z}(x)$ the {\em spectral index\/} of $x$ namely 
\[
\iMorse{Z}(x)\=-\spfl(\mathcal L_s, s\in [0,1]).
\]
Let $z_x$ be defined in  Equation \eqref{eq:z-x}. We define the {\em Maslov index of $z_x$\/} as the integer given by 
\begin{equation}
	\iMas{L_Z}(z_x)\=\iCLM\big(L_Z, \Graph \psi(t), t \in [0,T]\big)
\end{equation}
where $\psi$ denotes the fundamental solution of the Hamiltonian system given in Equation \eqref{eq:hamsys-finale}.
\end{defn}
\begin{prop}\label{thm:index}
	Under the previous notation, if $x$ is a critical point of $A_Z$, then  $\iMorse{Z}(x)$ is well-defined. Moreover the following equality holds
	\begin{equation}
		\iMorse{Z}(x)=\iMas{L_Z}(z_x).
	\end{equation}
\end{prop}
\proof 
For the proof of this result we refer the reader to  \cite[Theorem 2.5]{HS09}.
\qed

We close this section with the following characterization of the linear instability of a periodic orbit through the parity of the  degree index. 
\begin{thm}
Let $x$ be a $T$-periodic solution of Equation~\eqref{eq:bvp}.  If one of the following two alternatives hold
\begin{itemize}
\item[]{\bf (OR)} $x$   is  orientation preserving and $\icon(x) +n$   is odd
\item[]{\bf (NOR)} $x$   is non orientation preserving and $\icon(x) +n$  is even
\end{itemize}
then $x$ is linearly unstable.
\end{thm}
\proof
We prove the theorem in the case (OR) being the other completely analogous.  Let $x$ be a $T$-periodic orientation preserving solutions and let $z_x$ be the solution of the corresponding Hamiltonian system defined in Equation~\eqref{eq:z-x}. 

By Proposition~\ref{thm:index}, we get that 
\[
\iCLM(\Delta, \Graph \psi(t), t \in [0,T])= -\spfl(\mathcal L_s, s\in [0,1]).
\]
By  \cite[Equation (3.18)]{HS09}, we get that 
$\spfl(\mathcal L_s, s\in [0,1])=\spfl(\mathcal A_s, s\in [0,1])$ where $s\mapsto\mathcal A_s$ is the path of second order self-adjoint Fredholm operators on $L^2$ with dense domain 
\[
E^2([0,T], \R^n)=\Set{u \in H^2([0,T], \R^n)|u(0)=u(T) \textrm{ and } u'(0)=u'(T)}.
\]
By invoking Theorem \ref{thm:main}, then we get
\begin{multline}
\iCLM(\Delta, \Graph \psi(t), t \in [0,T]) +n= -\spfl(\mathcal L_s, s\in [0,1])+n\\=-\spfl(\mathcal A_s, s\in [0,1])+n= -\icon(\psi, R)+n.
\end{multline}
Now since the parity of $-\icon(\psi, R)+n$ and  $\icon(\psi, R)+n$ both coincides, the result directly follows by  \cite[Theorem 1]{PWY19}. This concludes the proof.

\qed

\appendix



\section{A recap on the Maslov index and spectral flow}\label{sec:sf}

In this section, we provide the basic definition and properties on the Maslov index and the spectral flow. Our basic references are 
\cite{BJP14a, BJP14b, KOP19, PWY19} and references therein.


\subsection{On the Maslov index}\label{subsec:Maslovindexpath}

Given a $2n$-dimensional (real) symplectic space $(V,\omega)$, a {\em 
Lagrangian 
subspace\/} of $V$ is an $n$-dimensional subspace $L \subset V$ such that $L = 
L^\omega$ where $L^\omega$ denotes the {\em symplectic orthogonal\/}, i.e. the 
orthogonal 
with respect to the symplectic structure. 
We denote by $ \Lagr= \Lagr(V,\omega)$ the {\em Lagrangian Grassmannian of 
$(V,\omega)$\/}, namely the set of all Lagrangian subspaces of $(V, \omega)$
\[
\Lagr(V,\omega)\=\Set{L \subset V| L= L^{\omega}}.
\]
It is well-known that $\Lagr(V,\omega)$ is a manifold. For each $L_0 \in \Lagr$, 
let 
\[
\Lagr^k(L_0) \= \Set{L \in \Lagr(V,\omega) | \dim\big(L \cap L_0\big) =k } 
\qquad k=0,\dots,n.
\]
Each $\Lagr^k(L_0)$ is a real compact, connected submanifold of codimension 
$k(k+1)/2$. The topological closure 
of $\Lagr^1(L_0)$  is the {\em Maslov cycle\/} that can be also described as 
follows
\[
 \Sigma(L_0)\= \bigcup_{k=1}^n \Lagr^k(L_0)
\]
The top-stratum $\Lagr^1(L_0)$ is co-oriented meaning that it has a 
transverse orientation. To be 
more precise, for each $L \in \Lagr^1(L_0)$, the path of Lagrangian subspaces 
$(-\delta, \delta) \mapsto e^{tJ} L$ cross $\Lagr^1(L_0)$ transversally, and as 
$t$ increases the path points to the transverse direction. Thus  the Maslov cycle is two-sidedly embedded in 
$\Lagr(V,\omega)$. Based on the topological properties of the Lagrangian 
Grassmannian manifold, 
it is possible to define a fixed endpoints homotopy invariant called {\em Maslov 
index\/}.

\begin{defn}\label{def:Maslov-index}
Let $L_0 \in \Lagr(V,\omega)$ and let $\ell:[0,1] \to \Lagr(V, \omega)$ be a 
continuous path. We 
define the {\em Maslov index\/} $\iCLM$ as follows:
\[
 \iCLM(L_0, \ell(t); t \in[a,b])\= \left[e^{-\varepsilon J}\, \ell(t): 
\Sigma(L_0)\right]
\]
where the right hand-side denotes the intersection number and $0 < \varepsilon 
<<1$.
\end{defn}
For further reference we refer the interested reader to \cite{CLM94} and references therein. 
\begin{rem}
 It is worth noticing that for $\varepsilon>0$ small enough, the Lagrangian 
subspaces 
 $e^{-\varepsilon J} \ell(a)$ and $e^{-\varepsilon J} \ell(b)$ are off the 
singular cycle. 
\end{rem}
One efficient way to compute the Maslov index, was introduced by authors in 
\cite{RS93} via 
crossing forms. Let $\ell$ be a $\mathscr C^1$-curve of Lagrangian subspaces 
such that 
$\ell(0)= L$ and let $W$ be a fixed Lagrangian subspace transversal to $L$. For 
$v \in L$ and 
small enough $t$, let $w(t) \in W$ be such that $v+w(t) \in \ell(t)$.  Then the 
form 
\[
 Q(v)= \dfrac{d}{dt}\Big\vert_{t=0} \omega \big(v, w(t)\big)
\]
is independent on the choice of $W$. A {\em crossing instant\/} for $\ell$ is an 
instant $t \in [a,b]$ 
such that $\ell(t)$ intersects $W$ nontrivially. At each crossing instant, we 
define the 
crossing form as 
\[
 \Gamma\big(\ell(t), W, t \big)= Q|_{\ell(t)\cap W}.
\]
A crossing is termed {\em regular\/} if the crossing form is non-degenerate. If 
$\ell$ is regular meaning that 
it has only regular crossings, then the Maslov index is equal to 
\begin{equation}\label{eq:iclm-crossings}
 \iCLM\big(W, \ell(t); t \in [a,b]\big) = \coiMor\big(\Gamma(\ell(a), W; a)\big)+ 
\sum_{a<t<b} 
 \sgn\big(\Gamma(\ell(t), W; t\big)- \iMor\big(\Gamma(\ell(b), W; b\big)
\end{equation}
where the summation runs over all crossings $t \in (a,b)$ and $\coiMor, \iMor$ 
are the dimensions  of 
the positive and negative spectral spaces, respectively and $\sgn\= 
\coiMor-\iMor$ is the  signature. 
(We refer the interested reader to \cite{LZ00} and \cite[Equation (2.15)]{HS09}). 
We close this section by 
recalling some useful 
properties of the Maslov index. \\
\begin{itemize}
\item[]{\bf Property I (Reparametrization invariance)\/}. Let $\psi:[a,b] \to 
[c,d]$ be a 
continuous and piecewise smooth function with $\psi(a)=c$ and $\psi(b)=d$, then 
\[
 \iCLM\big(W, \ell(t)\big)= \iCLM(W, \ell(\psi(t))\big). 
\]
\item[] {\bf Property II (Homotopy invariance with respect to the ends)\/}. For 
any $s \in [0,1]$, 
let $s\mapsto \ell(s,\cdot)$ be a continuous family of Lagrangian paths 
parametrised on $[a,b]$ and 
such that $\dim\big(\ell(s,a)\cap W\big)$ and $\dim\big(\ell(s,b)\cap W\big)$ 
are constants, then 
\[
 \iCLM\big(W, \ell(0,t);t \in [a,b]\big)=\iCLM\big(W, \ell(1,t); t \in 
[a,b]\big).
\]
\item[]{\bf Property III (Path additivity)\/}. If $a<c<b$, then
\[
 \iCLM\big(W, \ell(t);t \in [a,b]\big)=\iCLM\big(W, \ell(t); t \in [a,c]\big)+
 \iCLM\big(W, \ell(t); t \in [c,b]\big) 
\]
\item[]{\bf Property IV (Symplectic invariance)\/}. Let $\Phi:[a,b] \to \Sp(2n, 
\R)$. Then 
\[
 \iCLM\big(W, \ell(t);t \in [a,b]\big)= \iCLM\big(\Phi(t)W, \Phi(t)\ell(t); t 
\in [a,b]\big).
\]
\end{itemize}

%
%
%
%

In the standard symplectic space $(\R^{2n}, \omega)$ we denote by $J$ the standard symplectic matrix defined by $J=\begin{bmatrix} 0&-\Id\\ \Id &0\end{bmatrix}$. The symplectic form $\omega$ can be  represented with respect to the Euclidean product $\langle\cdot, \cdot\rangle$ by $J$ as follows $\omega(z_1,z_2)=\langle J z_1,z_2\rangle$ for every $z_1, z_2 \in \R^{2n}$.  We consider the 1-codimensional (algebraic) subvariety 
\[
\Sp(2n, \R)^{0}\=\{M\in \Sp(2n, \R)| \det(M-\Id)=0\} \subset \Sp(2n, \R)
\] 
and let us define
\[
\Sp(2n,\R)^{*}=\Sp(2n,\R)\backslash \Sp(2n, \R)^{0}=\Sp(2n,
\R)^{+}\cup
\Sp(2n, \R)^{-}
\]
where 
\begin{multline}
\Sp(2n,\R)^+\=\{M\in \Sp(2n,\R)| \det(M-\Id)>0\} \quad \textrm{ and }\\
\Sp(2n,\R)^- \=\{M\in \Sp(2n,\R)
|\det(M-\Id)<0\}.
\end{multline}

For any $M\in \Sp(2n,\R)^{0}$, $\Sp(2n,\R)^{0}$ is
co-oriented at the point $M$
 by choosing  as positive direction the direction determined by
 $\frac{d}{dt}Me^{tJ}|_{t=0}$ with $t\geq0$ sufficiently small.  We recall that $\Sp(2n,\R)^{+}$ and
$\Sp(2n,\R)^{-}$ are two path connected
 components of $\Sp(2n,\R)^{*}$
  which are simple connected in $\Sp(2n,\R)$. (For the proof of these facts we refer, for instance,  the interested reader to \cite[pag.58-59]{Lon02} and references therein).  Following authors in \cite[Definition 2.1]{LZ00} we start by recalling the following definition.
 \begin{defn}\label{def:Maslov-index-ok}
Let $\psi:[a,b]\rightarrow\Sp(2n,\R)$ be a continuous path. Then there exists an $\varepsilon>0$ such that for every $\theta\in[-\varepsilon,\varepsilon]\setminus\{0\}$, the matrices $\psi(a)e^{J\theta}$ and $\psi(b)e^{J\theta}$  lying both out of $\Sp(2n,\R)^0$ . We define the {\em $\iomega{1}$-index\/} or the {\em Maslov-type index\/} as follows
\begin{equation}
\iomega{1}(\psi)\=[e^{-J\varepsilon}\psi:\Sp(2n,\R)^{0}]
\end{equation}
where the (RHS) denotes the intersection number between the perturbed path $t\mapsto e^{-J\varepsilon} \psi(t)$ with the singular cycle $\Sp(2n,\R)^0$. 
\end{defn}
Through the parity of the $\iomega{1}$-index it is possible to locate  endpoints of the perturbed symplectic path  $t\mapsto e^{-J\varepsilon} \psi(t)$.  
 \begin{lem}\label{lem:parity property}{\bf (\cite[Lemma 5.2.6]{Lon02})\/}
 Let $\psi:[a,b] \to \Sp(2n, \R)$ be a continuous path.  The  following characterization holds
 \begin{itemize}
 \item[] $\iomega{1}(\psi)$ is even $\iff$ both the endpoints
 $e^{-\varepsilon J}\psi(a)$ and $e^{-\varepsilon J}\psi(b)$ lie in $\Sp(2n, \R)^+$ or
in $\Sp(2n, \R)^-$.
 \end{itemize}
 \end{lem}
We close this section with a rather technical  result which will be used in the proof of the main instability criterion. 
\begin{lem}{\bf (\cite[Lemma 3.2]{HS10}.
)\/}\label{prop:how to know in which component}
 Let $\psi:[a,b] \to \Sp(2n, \R)$ be a continuous symplectic path such that $\psi(0)$ is linearly stable.
 \begin{enumerate}
  \item  If $1 \notin \sigma\big(\psi(a)\big)$ then there exists $\varepsilon >0$
sufficiently small such that
  $\psi(s) \in \Sp(2n, \R)^+$ for $|s| \in (0, \varepsilon)$.
  \item We assume that $\dim \ker \big(\psi(a)-\Id\big)=m$ and
  $\trasp{\psi(a)}J \psi'(a)\vert_V$ is non-singular for $V\= \psi^{-1}(a) \R^{2m}$. If
  $\ind\big({\trasp{\psi(a)}J \psi'(a)\vert_V}\big)$ is even [resp. odd] then there exists $\delta>0$
  sufficiently small such that  $\psi(s) \in \Sp(2n, \R)^+$
  [resp. $\psi(s) \in \Sp(2n, \R)^-$] for $|s| \in (0, \delta)$.
  \end{enumerate}
\end{lem}
\begin{rem}
Knowing that $M \in \Sp^0(2n,\R)$, without any further information, it is not possible a priori to locate in which path connected components of $\Sp(2n, \R)^*$ is located the perturbed  matrix $e^{\pm \delta J}M$ for arbitrarily small positive  $\delta$. However if $M$ is linearly stable, we get the following result.
\end{rem} 
\begin{lem}\label{lem:linear stable is in positive side}
 Let $M\in \Sp(2n, \R) $ be a linearly stable symplectic matrix (meaning that $\spec{M} \subset \U$ and $M$ is diagonalizable). 
 Then, there exists $\delta >0$ sufficiently small such that $e^{\pm \delta J}M \in \Sp(2n,\R)^+$.
\end{lem}
\begin{proof}
Let us consider the  symplectic path pointwise defined by
 $M(\theta)\= e^{-\theta J}M$. By a direct computation we get that
\[
 \trasp{M(\theta)} J \dfrac{d}{d\theta} M(\theta)\Big\vert_{\theta=0} =
\trasp{M}M.
\]
We observe that $\trasp{M}M$ is symmetric and positive semi-definite; moreover since $M$ invertible it follows that
$\trasp{M}M$ is actually positive definite. Thus, in particular,
$\iiindex(\trasp{M}M)=0$. By invoking Lemma \ref{prop:how to know in which component} it follows
that there exists $\delta >0$ such that $M(\pm \delta) \in \Sp(2n, \R)^+ $. This concludes the proof.
\end{proof}

\subsection{On the spectral flow}\label{subsec:sf}
The aim of this subesection is to briefly recall the Definition and the main
properties of the spectral
flow for a continuous path of closed self-adjoint Fredholm operators  $\mathcal{CF}^{sa}(\mathcal H)$ on the Hilbert space $H$.
Our basic reference is \cite{HPY19}  and references therein.

Let $\mathcal H$ be a separable complex Hilbert space and let
$A: \mathcal D(A) \subset \mathcal H \to \mathcal H$ be  a  self-adjoint
Fredholm
operator. By the Spectral decomposition Theorem (cf., for instance,
\cite[Chapter III,
Theorem 6.17]{Kat80}), there is an orthogonal decomposition $
 \mathcal H= E_-(A)\oplus E_0(A) \oplus E_+(A),$
that reduces the operator $A$
and has the property that
\[
 \sigma(A) \cap (-\infty,0)=\sigma\big(A_{E_-(A)}\big), \quad
 \sigma(A) \cap \{0\}=\sigma\big(A_{E_0(A)}\big),\quad
 \sigma(A) \cap (0,+\infty)=\sigma\big(A_{E_+(A)}\big).
\]
\begin{defn}\label{def:essential}
Let $A \in \mathcal{CF}^{sa}(\mathcal H)$. We  term $A$ {\em essentially
positive\/}
if $\sigma_{ess}(A)\subset (0,+\infty)$, {\em essentially negative\/} if
$\sigma_{ess}(A)\subset (-\infty,0)$ and finally
{\em strongly indefinite\/} respectively if $\sigma_{ess}(A) \cap (-\infty,
0)\not=
\emptyset$ and $\sigma_{ess}(A) \cap ( 0,+\infty)\not=\emptyset$.
\end{defn}
\noindent
If $\dim E_-(A)<\infty$,
we define its {\em Morse index\/}
as the integer denoted by $\iindex{A}$ and defined as $
 \iindex{A} \= \dim E_-(A).$
Given $A \in\cfsa(\mathcal H)$, for  $a,b \notin
\sigma(A)$ we set
\[
\mathcal P_{[a,b]}(A)\=\Real\left(\dfrac{1}{2\pi\, i}\int_\gamma
(\lambda-A)^{-1} d\, \lambda\right)
\]
where $\gamma$ is the circle of radius $\frac{b-a}{2}$ around the point
$\frac{a+b}{2}$. We recall that if
$[a,b]\subset \sigma(A)$ consists of  isolated eigenvalues of finite type then
$
 \Imm \mathcal P_{[a,b]}(A)= E_{[a,b]}(A)\= \bigoplus_{\lambda \in (a,b)}\ker
(\lambda -A);
$
(cf. \cite[Section XV.2]{GGK90}, for instance) and $0$ either belongs in the
resolvent set of $A$ or it is an isolated eigenvalue of finite multiplicity.
Let us now consider the {\em graph distance topology\/} which is the topology
induced by the {\em gap
metric\/} $d_G(A_1, A_2)\=\norm{P_1-P_2}$
where $P_i$ is the orthogonal projection onto the graph of $A_i$ in the product space
$\mathcal H
\times \mathcal H$. The next result allow us to  define the spectral flow for
gap
continuous paths in  $\cfsa(\mathcal H)$.
\begin{prop}\label{thm:cor2.3}
 Let $A_0 \in \cfsa(\mathcal H)$ be fixed.
 \begin{enumerate}
  \item[(i)] There exists a positive real number $a \notin \sigma(A_0)$ and an
open neighborhood $\mathscr N \subset  \cfsa(\mathcal H)$ of $A_0$ in the gap
topology such that $\pm a \notin
\sigma(A)$ for all $A \in  \mathscr N$ and the map
 \[
  \mathscr N \ni A \longmapsto \mathcal P_{[-a,a]}(A) \in \Lin(\mathcal H)
 \]
is continuous and the projection $\mathcal P_{[-a,a]}(A)$ has constant finite
rank for all $t \in \mathscr N$.
 \item[(ii)] If $\mathscr N$ is a neighborhood as in (i) and $-a \leq c \leq d
\leq a$ are such that $c,d \notin
 \sigma(A)$ for all $A \in \mathscr N$, then $A \mapsto \mathcal P_{[c,d]}(A)$
is
continuous on $\mathscr N$.
 Moreover the rank of $\mathcal P_{[c,d]}(A) \in \mathscr N$ is finite and
constant.
 \end{enumerate}
\end{prop}
\begin{proof}
For the proof of this result we refer the interested reader to
\cite[Proposition 2.10]{BLP05}.
\end{proof}
Let $\mathcal A:[a,b] \to \cfsa(\mathcal H)$ be a gap continuous path.  As
consequence
of Proposition \ref{thm:cor2.3}, for every $t \in [a,b]$ there exists $a>0$ and
an open
connected neighborhood $\mathscr N_{t,a} \subset \cfsa(\mathcal H)$ of
$\mathcal
A(t)$
such that $\pm a \notin \sigma(A)$ for all $A\in \mathscr N_{t,a}$ and the map
$\mathscr N_{t,a} \in A \longmapsto \mathcal P_{[-a,a]}(A) \in \mathcal
B$
is continuous and hence $ \rk\left(\mathcal P_{[-a,a]}(A)\right)$ does not
depends on $A \in \mathscr N_{t,a}$. Let us consider the open covering
of the interval $[a,b]$ given by the
pre-images of the neighborhoods $\mathcal
N_{t,a}$ through $\mathcal A$ and, by choosing a sufficiently fine partition of
the interval $[a,b]$ having diameter less than the Lebesgue number
of the covering, we can find  $a=:t_0 < t_1 < \dots < t_n:=b$,
operators $T_i \in \cfsa(\mathcal H)$ and
positive real numbers $a_i $, $i=1, \dots , n$ in such a way the restriction of
the path $\mathcal A$ on the
interval $[t_{i-1}, t_i]$ lies in the neighborhood $\mathscr N_{t_i, a_i}$ and
hence the
$\dim E_{[-a_i, a_i]}(\mathcal A_t)$ is constant for $t \in [t_{i-1},t_i]$,
$i=1, \dots,n$.
\begin{defn}\label{def:spectral-flow-unb}
The \emph{spectral flow of $\mathcal A$} (on the interval $[a,b]$) is defined by
\[
 \spfl(\mathcal A, [a,b])\=\sum_{i=1}^N \dim\,E_{[0,a_i]}(\mathcal A_{t_i})-
 \dim\,E_{[0,a_i]}(\mathcal A_{t_{i-1}}) \in \Z.
\]
\end{defn}
(In shorthand Notation we  denote  $\spfl(\mathcal A, [a,b])$ simply  by
$\spfl(\mathcal A)$ if no confusion  is possible).
The spectral flow as given in Definition \ref{def:spectral-flow-unb} is
well-defined
(in the sense that it is independent either on the partition or on the $a_i$)
and only depends on
the continuous path $\mathcal A$. Here We list one of the useful properties of the spectral flow.
\begin{itemize}
 \item[]  {\bf (Path Additivity)\/} If $\mathcal A_1,\mathcal
A_2: [a,b] \to
 \cfsa(\mathcal H)$ are two continuous path such that
$\mathcal A_1(b)=\mathcal A_2(a)$, then
 $
  \spfl(\mathcal A_1 *\mathcal A_2) = \spfl(\mathcal A_1)+\spfl(\mathcal A_2).
 $
\end{itemize}
As already observed, the spectral flow, in general,  depends on the whole path
and not
just on the ends. However, if the path has a special form, it actually depends
on the
end-points. More precisely, let  $\mathcal A ,\mathcal B\in \cfsa(\mathcal H)$
and let $\widetilde{\mathcal A}:[a,b] \to \cfsa(\mathcal H)$ be the path
pointwise defined by $\widetilde{\mathcal A}(t)\=\mathcal A+ \widetilde{\mathcal
B}(t)$  where $
\widetilde{\mathcal B}$ is any continuous curve of $\mathcal A$-compact
operators parametrised on $[a,b]$
such that  $\widetilde{\mathcal B}(a)\=0$ and  $ \widetilde{\mathcal B}(b)\=
\mathcal B$. In this case,
the spectral flow depends of the
path $\widetilde A$, only on the endpoints (cfr. \cite{ZL99} and reference
therein).
\begin{rem}
 It is worth noticing that, since every operator $\widetilde{\mathcal A}(t)$ is
a compact perturbation of a
 a fixed one, the path $\widetilde{\mathcal A}$ is actually a continuous path
into $\Lin(\mathcal W; \mathcal H)$,
 where $\mathcal W\=\mathcal D(\mathcal A)$.
\end{rem}
\begin{defn}\label{def:rel-morse-index}(\cite[Definition 2.8]{ZL99}).
 Let $\mathcal A ,\mathcal B\in \cfsa(\mathcal H)$ and we assume that $\mathcal
B$ is
 $\mathcal A$-compact (in the sense specified above). Then the
{\em  relative Morse index of the pair $\mathcal A$, $\mathcal A+\mathcal B$\/}
is
defined by $
  \irel(\mathcal A, \mathcal A+\mathcal B)=-\spfl(\widetilde{\mathcal A};[a,b])$
where $\widetilde{\mathcal A}\=\mathcal A+ \widetilde{\mathcal B}(t)$ and where
$
\widetilde{\mathcal B}$ is any continuous curve parametrised on $[a,b]$
of $\mathcal A$-compact operators such that
$\widetilde{\mathcal B}(a)\=0$ and
$ \widetilde{\mathcal B}(b)\= \mathcal B$.
\end{defn}
\noindent
In the special case in which the Morse index of both operators $\mathcal A$ and
$\mathcal A+\mathcal B$ are
finite, then
\begin{equation}\label{eq:miserve}
\irel(\mathcal A, \mathcal A+\mathcal B)=\iindex{\mathcal A +\mathcal
B}-\iindex{\mathcal A}.
\end{equation}

Let  $\mathcal W, \mathcal H$ be separable Hilbert spaces with a dense and
continuous
inclusion $\mathcal W \hookrightarrow \mathcal H$ and let
$\mathcal A:[a,b] \to \cfsa(\mathcal H)$  having fixed domain $\mathcal W$. We
assume that $\mathcal A$ is
a continuously differentiable path  $\mathcal A: [a,b] \to \cfsa(\mathcal H)$
and
we denote by $\dot{\mathcal A}_{\lambda_0}$ the derivative of
$\mathcal A_\lambda$ with respect to the parameter $\lambda \in [a,b]$ at
$\lambda_0$.
\begin{defn}\label{def:crossing-point}
 An instant $\lambda_0 \in [a,b]$ is called a {\em crossing instant\/} if
$\ker\, \mathcal A_{\lambda_0} \neq 0$. The
 crossing form at $\lambda_0$ is the quadratic form defined by
\begin{equation}
 \Gamma(\mathcal A, \lambda_0): \ker \mathcal A_{\lambda_0} \to \R, \quad
\Gamma(\mathcal A, \lambda_0)[u] =
\langle \dot{\mathcal A}_{\lambda_0}\, u, u\rangle_\mathcal H.
\end{equation}
Moreover a  crossing $\lambda_0$ is called {\em regular\/}, if $\Gamma(\mathcal
A, \lambda_0)$ is non-degenerate.
\end{defn}
We recall that there exists $\varepsilon >0$ such that   $\mathcal A +\delta \,
\Id_\mathcal H$ has only regular crossings
  for almost every $\delta \in (-\varepsilon, \varepsilon)$. In the special case in which all crossings are regular, then the spectral flow
can be easily computed through  the
crossing forms. More precisely the following result  holds.
\begin{prop}\label{thm:spectral-flow-formula}
 If $\mathcal A:[a,b] \to \cfsa(\mathcal W, \mathcal H)$ has only regular
crossings then they are in a finite
 number and
 \[
  \spfl(\mathcal A, [a,b]) = -\mathrm{n_-}{\left[\Gamma(\mathcal A,a)\right]}+
\sum_{t_0 \in (a,b)}
  \sgn\left[\Gamma(\mathcal A, t_0)\right]+
  \coiindex{\Gamma(\mathcal A,b)}
 \]
where the sum runs over all the crossing instants.
\end{prop}
\begin{proof}
 The proof of this result follows by arguing as in \cite{RS95}. This conclude
the proof.
\end{proof}


\vspace{1cm}
\noindent
\textsc{Prof. Alessandro Portaluri}\\
DISAFA\\
Università degli Studi di Torino\\
Largo Paolo Braccini 2 \\
10095 Grugliasco, Torino\\
Italy\\
Website: \url{https://sites.google.com/view/alessandro-portaluri/}\\
E-mail: \email{alessandro.portaluri@unito.it}

\vspace{1cm}
\noindent
\textsc{Prof. Li Wu}\\
Department of Mathematics\\
Shandong University\\
Jinan,Shandong, 250100\\
The People's Republic of China \\
China\\
E-mail: \email{vvvli@sdu.edu.cn}

\end{document}